\documentclass{article}
\usepackage[normalem]{ulem}
\usepackage{graphicx} 
\usepackage{amssymb}
\usepackage{amsthm}
\usepackage{bbm}
\usepackage{enumerate}
\usepackage[clock]{ifsym}
\usepackage{mathtools}
\usepackage{mathrsfs}
\usepackage[english]{babel}
\usepackage{mathrsfs}
\usepackage{fullpage}
\usepackage{xcolor}
\usepackage{hyperref}
\usepackage{tikz}
\usetikzlibrary{decorations.markings, arrows.meta}
\usepackage{pgfplots}
\pgfplotsset{compat=newest}
\usepgfplotslibrary{dateplot}
\usepgfplotslibrary{colormaps}

\usepackage[many]{tcolorbox}

\newtheorem{thm}{Theorem}[section]
\newtheorem{lm}[thm]{Lemma}

\newtheorem{defn}[thm]{Definition}
\newtheorem{prop}[thm]{Proposition}

\newtheorem{rmk}[thm]{Remark}

\numberwithin{equation}{section}

\newcommand{\argmax}{\textrm{argmax}}
\newcommand{\bbB}{\mathbb{B}}
\newcommand{\bfB}{\mathbf{B}}
\newcommand{\bfD}{\mathbf{D}}
\newcommand{\bfH}{\mathbf{H}}

\newcommand{\bfT}{\mathbf{T}}
\newcommand{\bfX}{\mathbf{X}}

\newcommand{\complexC}{\mathbb{C}}

\newcommand{\density}{p}
\newcommand{\dequal}{\stackrel{d}{=}}

\newcommand{\dL}{\mathcal{L}}

\newcommand{\betasmin}{\beta_{\min{}}}
\newcommand{\betasmax}{\beta_{\max{}}}
\newcommand{\betaemin}{\beta'_{\min{}}}
\newcommand{\betaemax}{\beta'_{\max{}}}
\newcommand{\alphamin}{\alpha_{\min{}}}
\newcommand{\alphamax}{\alpha_{\max{}}}

\newcommand{\FGUE}{F_{\mathrm{GUE}}}
\newcommand{\fGUE}{f_{\mathrm{GUE}}}

\newcommand{\hKPZ}{\mathcal{H}^{\mathrm{KPZ}}}
\newcommand{\identity}{\mathbf{1}}

\newcommand{\inn}{\mathrm{in}}

\newcommand{\prob}{\mathbb{P}}
\newcommand{\out}{\mathrm{out}}
\newcommand{\rC}{\mathrm{C}}
\newcommand{\rd}{\mathrm{d}}
\newcommand{\rD}{\mathrm{D}}
\newcommand{\realR}{\mathbb{R}}

\newcommand{\rH}{\mathrm{H}}
\newcommand{\rI}{\mathrm{I}}

\newcommand{\rmi}{\mathrm{i}}

\newcommand{\rL}{\mathrm{L}}

\newcommand{\rp}{\mathrm{p}}
\newcommand{\rR}{\mathrm{R}}

\newcommand{\UTfield}{\mathcal{H}_0^{\mathrm{UT}}}
\newcommand{\LUTfield}{\mathcal{H}^{\mathrm{UT}}}

\renewcommand{\vec}[1]{\ensuremath{\boldsymbol{#1}}}

\def\fddto{\xrightarrow{\textit{f.d.d.}}}
\def\law{{\rm Law}}

\newcommand{\rS}{\mathrm{S}}

\newcommand{\cA}{\mathcal{A}}

\title{Limiting one-point fluctuations of the geodesic in the directed landscape near the endpoints when the geodesic length goes to infinity}
\author{Zhipeng Liu\footnote{Department of Mathematics, University of Kansas, Lawrence, KS 66045. Email: \texttt{zhipeng@ku.edu}} \and Chenyang Ma\footnote{Department of Mathematics, University of Kansas, Lawrence, KS 66045. Email: \texttt{c175m266@ku.edu}} \and Tejaswi Tripathi\footnote{Department of Mathematics, University of Kansas, Lawrence, KS 66045. Email: \texttt{tejaswit@ku.edu}}}
\date{}

\begin{document}

\maketitle

\begin{abstract}
    We consider the limiting fluctuations of the geodesic in the directed landscape, conditioning on its length going to infinity. It was shown in \cite{Liu22b,Ganguly-Hegde-Zhang23} that when the directed landscape $\mathcal{L}(0,0;0,1) = L$ becomes large, the geodesic from $(0,0)$ to $(0,1)$ lies in a strip of size $O(L^{-1/4})$ and behaves like a Brownian bridge if we zoom in the strip by a factor of $L^{1/4}$. Moreover, the length along the geodesic with respect to the directed landscape fluctuates of order $O(L^{1/4})$ and its limiting one-point distribution is Gaussian \cite{Liu22b}. In this paper, we further zoom in a  smaller neighborhood of the endpoints when $\mathcal{L}(0,0;0,1) = L$ or $\mathcal{L}(0,0;0,1) \ge L$, and show that there is a critical scaling window $L^{-3/2}:L^{-1}:L^{-1/2}$ for the time, geodesic location, and geodesic length, respectively. Within this scaling window, we find a nontrivial limit of the one-point joint distribution of the geodesic location and length as $L\to\infty$. This limiting distribution, if we tune the time parameter to infinity, converges to the joint distribution of two independent Gaussian random variables, which is consistent with the results in \cite{Liu22b}. We also find a surprising connection between this limiting distribution and the one-point distribution of the upper tail field of the KPZ fixed point recently obtained in \cite{Liu-Zhang25}.
\end{abstract}

\section{Introduction}

\subsection{Motivation and main results}

The directed landscape $\dL(x,s;y,t)$, defined on $\realR_{\uparrow}^4:=\{(x,s;y,t): x,y,s,t\in\realR, s<t\}$, is a random ``metric'' constructed in \cite{Dauvergne-Ortmann-Virag22} as a limiting field of the Brownian last passage percolation. It is believed, and in some cases proved, to be a universal limit for models in the Kardar-Parisi-Zhang universality class \cite{Dauvergne-Virag21,Wu23,Aggarwal-Corwin-Hegde24,Dauvergne-Zhang24}. The directed landscape $\dL(x,s;y,t)$ satisfies the following reverse triangle inequality 
\begin{equation}
\dL(x,s;x',s') + \dL(x',s';y,t) \le \dL(x,s; y,t),\quad \text{ for all } x,x',y,s,s',t\in\realR, \text{ satisfying }s<s'<t.
\end{equation}
This induces the concept of the geodesic $\pi_{x,s;y,t}: [s,t]\to \realR$ for any $(x,s;y,t)\in\realR_{\uparrow}^4$, such that $\pi_{x,s;y,t}(s) = x$, $\pi_{x,s;y,t}(t) =y$, and 
\begin{equation}
    \label{eq:def_geodesic}
\dL (x,s; \pi_{x,s;y,t}(s'),s') + \dL(\pi_{x,s;y,t}(s'),s';y,t) = \dL(x,s;y,t),\quad \text{for all }s'\in(s,t).
\end{equation}
It is known that such a geodesic $\pi_{x,s;y,t}$ for fixed $(x,s;y,t)\in\realR_{\uparrow}^4$ exists and is unique almost surely \cite{Dauvergne-Ortmann-Virag22}. One simple property of a geodesic is that any segment of the geodesic is also a geodesic between the two endpoints. Hence, we call $\dL(\pi(s'),s';\pi(t'),t')$ the \emph{length} of the geodesic segment from $(\pi(s'),s')$ to $(\pi(t'),t')$ for all $s'$ and $t'$ satisfying $s\le s'<t'\le t$. 

The directed landscape and its geodesics have very nice shifting and scaling invariance properties, which allow a translation of the properties of a geodesic $\pi_{x,s;y,t}$ to that of another geodesic $\pi_{x',s';y',t'}$ with different parameters. See Lemma 10.2 in \cite{Dauvergne-Ortmann-Virag22} for these properties. Without loss of generality, we consider the geodesic $\pi_{0,0;0,1}$ in this paper and denote
\begin{equation}
\pi(t) = \pi_{0,0;0,1}(t),\quad 0\le t\le 1
\end{equation}
as the geodesic location, and 
\begin{equation}
\dL(t) = \dL(0,0; \pi(t),t),\quad 0<t\le 1
\end{equation}
as the length of the geodesic segment from $(0,0)$ to $(\pi(t),t)$. Note that $\dL(1) = \dL(0,0;0,1)$. We also set $\dL(0)=0$ for convenience. 

\bigskip

We are mainly interested in the limiting fluctuations of the geodesic $\pi(t)$ and the length $\dL(t)$, conditioning on $\dL(1)=L$ becoming very large, which is called the upper tail conditioning. There have been quite a number of recent results on the large deviation of the directed landscape and related objects in the upper tail regime \cite{Basu-Ganguly19,Ganguly-Hedge22,Lamarre-Yves_Lin_Tsai2023,Lin-Tsai25,Das-Dauvergne-Virag24,Das-Tsai2024}, and their limiting conditional fluctuations \cite{Liu22b,Liu-Wang24,Ganguly-Hegde-Zhang23,Baik-Liu2024,Liu-Zhang25,Baik-Cordaro-Tripathi2025}. Particularly, the fluctuations of the geodesic with the upper tail conditioning were considered in \cite{Liu22b,Liu-Wang24,Ganguly-Hegde-Zhang23}. It was shown in \cite{Liu22b} that for any fixed $t\in(0,1)$, when $L$ becomes large, $\pi(t)$ fluctuates of order $O(L^{-1/4})$, and $\dL(t)$ fluctuates of order $O(L^{1/4})$. Furthermore, $\pi(t)$ and $\dL(t)$ after rescaling converge to two independent Gaussian random variables. More explicitly, for any fixed $x_1,x_2,h_1,h_2 \in \realR$ satisfying $x_1<x_2$ and $h_1<h_2$,
\begin{equation}
\label{eq:GaussianLimit_OnePoint}
\prob \left( \frac{2 L^{1/4} \pi(t)}{\sqrt{t(1-t)}} \in (x_1, x_2), \frac{\dL(t) -tL}{\sqrt{t(1-t)}L^{1/4}} \in (h_1,h_2) \ \big| \ \dL(1)=L \right)  \to \int_{x_1}^{x_2} \int_{h_1}^{h_2} \phi(x)\phi(h) \rd h \rd x
\end{equation}
as $L\to\infty$, where $\phi(x):= \frac{1}{\sqrt{2\pi}} e^{-x^2/2}$ is the standard Gaussian density. Later, based on \eqref{eq:GaussianLimit_OnePoint} and a heuristic argument, \cite{Liu-Wang24} conjectured that 
\begin{equation}
\label{eq:conjecture_TwoBrownianBridges}
 \law\left( \left\{2 L^{1/4} \pi(t), L^{-1/4}(\dL(t)-tL)\right\}_{t\in(0,1)}\bigg|\ \dL(1)=L \right) \to  \law \left(\left\{\bbB_1(t), \bbB_2(t)\right\}_{t\in(0,1)}\right)
\end{equation}
as $L\to\infty$, where $\bbB_1$ and $\bbB_2$ are two independent standard Brownian bridges. The Brownian bridge limit for the geodesic location was also suggested for the Brownian polymer model under the upper tail conditioning in an earlier paper \cite{Basu-Ganguly19}. Recently, \cite{Ganguly-Hegde-Zhang23} were able to show that the geodesic location after appropriate scaling indeed converges to a Brownian bridge under the upper tail conditioning for the directed landscape. This settled the first half of the conjecture \eqref{eq:conjecture_TwoBrownianBridges}, while the full conjecture remains an open problem.

\bigskip

\bigskip

While the conjecture \eqref{eq:conjecture_TwoBrownianBridges} and the results \cite{Liu22b,Liu-Wang24,Ganguly-Hegde-Zhang23} mentioned above are mainly about the Gaussian-type fluctuations of the geodesic when the time parameter is of order $O(1)$ under the upper tail conditioning, it is natural to ask about the local behavior of the geodesic under the same conditioning. It is expected, from an analogous result for the KPZ fixed point in \cite{Liu-Zhang25}, that there is a critical scaling window when the time parameter is of order $O(L^{-3/2})$ distance from $0$ or $1$, and the geodesic location fluctuates of order $O(L^{-1})$, conditioned on $\dL(1)=L$ as $L\to\infty$. Our main result in this paper is the following theorem, which provides the limiting joint distribution of the geodesic location and length in this window. 

\begin{thm}
\label{thm:main}
Assume $t>0$ is fixed. For any $x_1,x_2,h_1,h_2\in\realR$ satisfying $x_1<x_2$ and $h_1<h_2$, we have 
\begin{equation}
\label{eq:newmain0}
\prob\left( L \pi (tL^{-3/2}) \in (x_1,x_2), L^{1/2}\dL(tL^{-3/2}) \in (h_1,h_2) \ \big| \ \dL(1)\ge L \right) \to \int_{h_1}^{h_2} \int_{x_1}^{x_2}  \rp(h,x;t) \rd x \rd h,
\end{equation}
and
\begin{equation}
\label{eq:newmain}
\prob\left( L \pi (1-tL^{-3/2}) \in (x_1,x_2), L^{1/2}(L-\dL(1-tL^{-3/2})) \in (h_1,h_2) \ \big| \ \dL(1)\ge L \right) \to \int_{h_1}^{h_2} \int_{x_1}^{x_2} \hat \rp(h,x;t) \rd x \rd h
\end{equation}
as $L\to\infty$. We also have
\begin{equation}
\label{eq:main0}
\prob\left( L \pi (tL^{-3/2}) \in (x_1,x_2), L^{1/2}\dL(tL^{-3/2}) \in (h_1,h_2) \ \big| \ \dL(1)=L \right) \to \int_{h_1}^{h_2} \int_{x_1}^{x_2} \rp(h,x;t) \rd x \rd h,
\end{equation}
and
\begin{equation}
\label{eq:main}
\prob\left( L \pi (1-tL^{-3/2}) \in (x_1,x_2), L^{1/2}(L-\dL(1-tL^{-3/2})) \in (h_1,h_2) \ \big| \ \dL(1)=L \right) \to \int_{h_1}^{h_2} \int_{x_1}^{x_2} \rp(h,x;t) \rd x \rd h,
\end{equation}
as $L\to\infty$. Here  the functions $\hat\rp$ and $\rp$ are explicitly defined in Definition \ref{def:rp}.
\end{thm}
\begin{rmk}
Note that the limits in \eqref{eq:newmain0} and \eqref{eq:newmain} are different. Heuristically, this difference naturally arises from the sensitivity of the two variables, $L^{1/2}\dL(tL^{-3/2})$ and $L^{1/2}(L-\dL(1-tL^{-3/2}))$, to an $O(L^{-1/2})$ perturbation of $L$. Such a perturbation is not visible in $L^{1/2}\dL(tL^{-3/2})$, but makes an $O(1)$ difference in $L^{1/2}(L-\dL(1-tL^{-3/2}))$ when $L$ becomes large. 

On the other hand, since $(\pi(s),\dL(s), \dL(1))$ has the same law as $(\pi(1-s),\dL(1)-\dL(1-s), \dL(1))$, the right hand sides of \eqref{eq:main0} and \eqref{eq:main} are the same.
\end{rmk}

We remark that both functions $\hat \rp$ and $\rp$ are joint probability density functions, as stated in the Proposition \ref{prop:joint_density} below. We do not have a direct proof of this proposition using the formulas of these two functions. Instead, we will provide a proof mainly using probabilistic arguments. The proof   is given in Section \ref{sec:proof_density}. 
\begin{prop}
    \label{prop:joint_density}
    Both functions $\hat \rp$ and $\rp$ are joint probability density functions for each $t>0$, i.e., $\hat\rp(h,x;t)\ge 0$ and $\rp(h,x;t)\ge 0$ for all $(h,x)\in\realR^2$, and 
\begin{equation}
\int_{\realR} \int_{\realR}\hat\rp(h,x;t)\rd x\rd h= \int_{\realR} \int_{\realR} \rp(h,x;t)\rd x \rd h =1.
\end{equation}
\end{prop}

\bigskip

Using the formulas of $\hat\rp$ and $\rp$, we are able to prove the following properties of these two joint probability density functions.

\begin{prop}
\label{prop:properties_p}
The functions $\hat \rp$ and $\rp$ have the following properties.
\begin{enumerate}[(1)]

\item Both functions are symmetric in $x$: $ \hat \rp(h,x;t) =\hat \rp(h,-x;t) $, and $\rp(h,x;t) =\rp(h,-x;t)$.

\item The following identities hold for any fixed $(h,x)\in\realR^2$ and $t>0$, 
\begin{equation}
\label{eq:relation_two_densities1}
\hat\rp(h,x;t) = \int_0^\infty \rp(h'+ h,x;t)\cdot 2e^{-2h'}\rd h', 
\end{equation}
and
\begin{equation}
\label{eq:relation_two_densities2}
 \rp(h,x;t) = -e^{2h} \frac{\partial}{\partial h} \left(\frac12e^{-2h} \hat\rp(h,x;t)\right).
\end{equation}

As a corollary, if $\rp(h,x;t)$ is the joint density of a pair of random variables $H$ and $X$, and $\hat \rp(h,x;t)$ is the joint density of $\hat H$ and $\hat X$, then $(\hat H,\hat X) \stackrel{d}{=}  (H, X) - (E,0)$, where $E$ is an exponential random variable with parameter $2$ that is independent of $H$ and $X$, and $\dequal$ denotes an equation in law.

\item When the time parameter goes to infinity, the joint density function $\rp$ converges to the product of two standard Gaussian densities after appropriate scaling
\begin{equation}
\label{eq:large_time_asymptotics}
\lim_{t\to\infty} \frac{t}{2} \rp\left(t + h\sqrt{t}, \frac{1}{2}x\sqrt{t}; t\right) = \phi(h)\phi(x),\quad (h,x)\in\realR^2,
\end{equation}
where $\phi$ is the standard Gaussian density function. The same statement holds for $\hat\rp$.
\item For any fixed $x\in\realR$ and $t>0$, we have the following approximations of the joint density functions $\rp$ and $\hat\rp$
\begin{equation}
\label{eq:right_tail_rp_h}
\begin{split}
    \rp(h,x;t) &= \frac{1}{4\pi t}e^{-\frac{4}{3}\frac{\left( h + \frac{x^2}{t}\right)^{3/2}}{t^{1/2}} +2h - \frac{2}{3}t}(1+o(1)),\\
    \hat\rp(h,x;t) &= \frac{1}{4\pi (ht)^{1/2}}e^{-\frac{4}{3}\frac{\left( h + \frac{x^2}{t}\right)^{3/2}}{t^{1/2}} +2h - \frac{2}{3}t}(1+o(1))
    \end{split}
\end{equation}
as $h\to\infty$. Moreover, for any fixed $h\in\realR$ and $t>0$, we have
\begin{equation}
\label{eq:right_tail_rp_x}
\begin{split}
\rp(h,x;t)&= \frac{1}{2\pi |x|}e^{-\frac{4}{3}\frac{\left( h + \frac{x^2}{t}\right)^{3/2}}{t^{1/2}} + 2h - \frac{2}{3}t}(1+o(1)),\\
\hat\rp(h,x;t)&= \frac{t}{2\pi x^2}e^{-\frac{4}{3}\frac{\left( h + \frac{x^2}{t}\right)^{3/2}}{t^{1/2}} + 2h - \frac{2}{3}t}(1+o(1))
\end{split}
\end{equation}
as $|x|\to\infty$. 

\end{enumerate}
\end{prop}
The proof of Proposition \ref{prop:properties_p} is given in Section \ref{sec:proof_proposition2}.

\begin{rmk}
     The appearance of the exponential random variable $E$ in Proposition \ref{prop:properties_p} (2) coincides with $\LUTfield(0,0)$, where $\LUTfield$ is called the upper tail field, a random field on $\realR^2$ defined in \cite{Liu-Zhang25}. The upper tail field is a limiting local field of the KPZ fixed point within the same conditioning and the same rescaling window as in this paper. See equations \eqref{eq:UpperTailFieldConvergence}, \eqref{eq:UpperTailFieldConvergence1}, and \eqref{eq:relation_two_UTfields}, and the related discussions in Section \ref{sec:KPZ}.
\end{rmk}

\begin{rmk}
    Heuristically, the asymptotics \eqref{eq:large_time_asymptotics} is consistent with  \eqref{eq:GaussianLimit_OnePoint} and \eqref{eq:conjecture_TwoBrownianBridges}. In fact, if we formally write 
    \begin{equation}
        \pi(\tau)\approx \frac12L^{-1/4}\bbB_1(\tau),\quad \dL(\tau)\approx \tau L + L^{1/4}\bbB_2(\tau)
    \end{equation}
    conditioned on $\dL(1)=L$, where $\bbB_1,\bbB_2$ are two independent Brownian bridges in \eqref{eq:conjecture_TwoBrownianBridges}, then
    heuristically we have, by choosing $\tau=s\epsilon$ with $\epsilon=tL^{-3/2}$ in the equation above, 
    \begin{equation}
        \begin{split}
            &\left(\bfX_s=\frac{2}{\sqrt{t}}L\pi(stL^{-3/2}),\bfH_s=\frac{1}{\sqrt{t}}\left(L^{1/2}\dL(stL^{-3/2})-st\right)\right) \\
            &
            \approx \left(\bfB_1(s)\approx\epsilon^{-1/2}\bbB_1(s\epsilon),  \bfB_2(s)\approx\epsilon^{-1/2}\bbB_2(s\epsilon)\right),
        \end{split}
    \end{equation}
    where $\bfB_1$ and $\bfB_2$ are two independent Brownian motions. When $s=1$, it is consistent with \eqref{eq:large_time_asymptotics}.
\end{rmk}

\subsection{Definition of the joint density functions $\hat\rp$ and $\rp$}
\label{sec:Def_rp}

\bigskip

The functions $\hat\rp$ and $\rp$ are defined as series expansions, with each term expressed as contour integrals on the complex plane $\complexC$. We first introduce these contours. Let  $\Gamma_\rL$ be a contour on the left half plane which starts from $\infty e^{-2\pi\rmi/3}$, and passes a finite point on the interval $(-1,0)$, then goes to $\infty e^{2\pi\rmi/3}$. We denote $\Gamma_\rR = -\Gamma_\rL$ with an orientation from $\infty e^{-\pi\rmi/3}$ to $\infty e^{\pi\rmi/3}$. See Figure \ref{fig:Gamma_Contours} for an illustration.
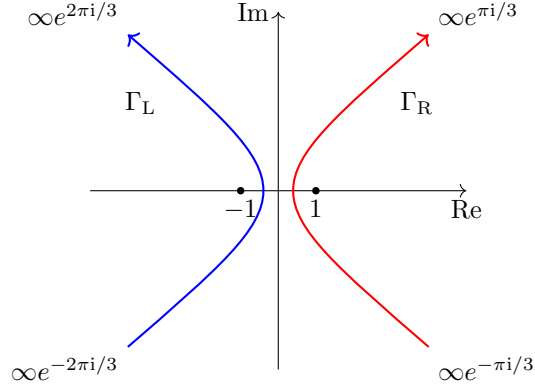
\begin{figure}
\centering
\begin{tikzpicture}[
    x=0.5cm, y=0.3cm 
    ]
    \def\R{8} 
    \def\angle{120} 
    
    \path 
        ({\R * cos(\angle)}, {-\R * sin(\angle)}) coordinate (A) 
        ({\R * cos(\angle)}, {\R * sin(\angle)}) coordinate (C); 
    \path 
        ({\R * cos(60)}, {-\R * sin(60)}) coordinate (D) 
        ({\R * cos(60)}, {\R * sin(60)}) coordinate (F); 
        
    \draw[->] ({\R * cos(\angle) - 1}, 0) -- ({-\R * cos(\angle) + 1}, 0) node[below] {Re};
    \draw[->] (0, {-\R * sin(\angle) - 1}) -- (0, {\R * sin(\angle) + 1}) node[left] {Im};
    
    \draw[blue, thick, ->] 
        (A) 
        .. controls (0.8, 0) .. (C) 
        node[midway, right, black] {}; 
    \draw[red, thick, ->] 
        (D) 
        .. controls (-0.8,0) .. (F) 
        node[midway, right, black] {}; 

    \node[below left] at (A) {$\infty e^{-2\pi \rmi/3}$};
    \node[above left] at (C) {$\infty e^{2\pi \rmi/3}$};
    \node[above left] at (-3,3) {$\Gamma_\rL$}; 
       \node[below right] at (D) {$\infty e^{-\pi \rmi/3}$};
       \node[above right] at (3,3) {$\Gamma_\rR$};
       \node[below] at (-1,0)  {$-1$};
       \node[circle,fill,inner sep=1pt] at (-1,0) {};
       \node[circle,fill,inner sep=1pt] at (1,0) {};
       \node[below] at (1,0) {$1$};
    \node[above right] at (F) {$\infty e^{\pi \rmi/3}$};
\end{tikzpicture}
\caption{Illustration of the contours $\Gamma_\rL$ and $\Gamma_\rR$. Note that the $\Gamma_\rL$ contour lies between the two points $-1$ and $0$, and the $\Gamma_\rR$ contour lies between $0$ and $1$.}
\label{fig:Gamma_Contours}
\end{figure}

We define the following Cauchy determinant 
\begin{equation}
\label{eq:def_CauchyDet}
\rC(\vec{W};\vec{W'}):= \det\left[ \frac{1}{w_i- w'_{j}} \right]_{i,j=1}^n = (-1)^{n(n-1)/2} \frac{ \prod_{1\le i<j\le n} (w_i-w_j)(w'_i-w'_j)}{ \prod_{i=1}^n \prod_{j=1}^n (w_i-w'_j)}
\end{equation}
for all $n\ge 1$ and vectors $\vec{W}=(w_1,\ldots,w_n), \vec{W'}=(w'_1,\ldots, w'_n)\in\complexC^n$ satisfying $w_i\ne w'_j$ for all $1\le i,j\le n$. Note that the dimensions of $\vec{W}$ and $\vec{W'}$ have to be equal in the definition of the Cauchy determinant \eqref{eq:def_CauchyDet}.

Moreover, we introduce the following operation $\sqcup$ of two vectors
\begin{equation}
\label{eq:def_sqcup}
\vec{W} \sqcup \vec{W'}=(w_1,\ldots,w_n,w'_1,\ldots,w'_{n'}) \in\complexC^{n+n'}
\end{equation}
for all $\vec{W}=(w_1,\ldots,w_n)\in \complexC^n$ and $\vec{W'}=(w'_1,\ldots,w'_{n'}) \in\complexC^{n'}$. Especially, when $n'=1$ and $\vec{W'}=(1)$ or $\vec{W'}=(-1)$, we write $\vec{W}\sqcup (1) = (w_1,\ldots,w_n,1)$ and $\vec{W}\sqcup(-1)=(w_1,\ldots,w_n,-1)$.

We also define the following functions of two vectors
\begin{equation}
\label{eq:def_rS}
\rS_k(\vec{W};\vec{W'}):=\sum_{i=1}^n w_i^k -\sum_{i'=1}^{n'}(w'_{i'})^{k}, \quad k=1,2,3
\end{equation}
and 
\begin{equation}
\label{eq:def_rH}
\rH(\vec{W};\vec{W'}):= \frac{1}{12}(\rS_1(\vec{W};\vec{W'}))^4 + \frac{1}{4} (\rS_2(\vec{W};\vec{W'}))^2 -\frac13\rS_1(\vec{W};\vec{W'})\rS_3(\vec{W};\vec{W'}),
\end{equation}
where $n,n'\ge 1$ and $\vec{W}=(w_1,\ldots,w_n)\in\complexC^n$ and $\vec{W'}=(w'_1,\ldots,w'_{n'})\in\complexC^{n'}$.

Finally, we define the following differential operator
\begin{equation}
\label{eq:def_DiffOperator}
\rD=\rD_{h,x;t}:= \frac{1}{12}\partial_h^4 +\frac14\partial_x^2 +\partial_h\partial_t ,
\end{equation}
where $\partial_h$, $\partial_x$ and $\partial_t$ are partial differential operators with respect to the parameters $h, x$ and $t$ respectively.

\begin{defn}
\label{def:rp}

Assume $t>0$. We define the function $\rp(h,x;t)$ and $\hat\rp(h,x;t)$ by the following series expansion formulas
\begin{equation}
\label{eq:def_rp}
\rp(h,x;t) = \sum_{n=1}^\infty \frac{1}{(n!)^2} \int_{\Gamma_\rL^n\times \Gamma_\rR^n}  \rI_n(\vec{U};\vec{V}) \prod_{i=1}^n\frac{\rd u_i}{2\pi\rmi} 
\prod_{i=1}^n\frac{\rd v_i}{2\pi\rmi}
\end{equation}
and
\begin{equation}
\label{eq:def_hat_rp01}
\hat\rp(h,x;t) = \int_0^\infty \rp(h'+ h,x;t)\cdot 2e^{-2h'}\rd h',
\end{equation} or equivalently, by changing the order of integration and summation,
\begin{equation}
\label{eq:def_hat_rp}
\hat \rp(h,x;t) =  \sum_{n=1}^\infty \frac{1}{(n!)^2} \int_{\Gamma_\rL^n\times \Gamma_\rR^n} \hat \rI_n(\vec{U};\vec{V}) \prod_{i=1}^n\frac{\rd u_i}{2\pi\rmi} 
\prod_{i=1}^n\frac{\rd v_i}{2\pi\rmi}.
\end{equation}
The integrands $\rI_n(\vec{U};\vec{V})$ and $\hat\rI_n(\vec{U};\vec{V})$ for $\vec{U}=(u_1,\ldots,u_n)\in\complexC^n$ and $\vec{V}=(v_1,\ldots,v_n)\in\complexC^n$ are defined by
\begin{equation}
\label{eq:def_rI}
\rI_n(\vec{U};\vec{V}) =-2\cdot \rC(\vec{U};\vec{V})\cdot \rC(\vec{V}\sqcup(-1);\vec{U}\sqcup(1))\cdot \rD\left( e^{-\frac{2}{3}t+2h}\prod_{\ell=1}^n \frac{f_{h,x;t}(u_\ell)}{f_{h,x;t}(v_\ell)}\right)
\end{equation}
and
\begin{equation}
\label{eq:def_hat_rI}
 \hat\rI_n(\vec{U};\vec{V})= \frac{2}{\sum_{i=1}^n (-u_i+v_i)}\cdot \rI_n(\vec{U};\vec{V}),
\end{equation}
where $\rC$ denotes the Cauchy determinant as defined in \eqref{eq:def_CauchyDet}, $\rD$ denotes the differential operator defined in \eqref{eq:def_DiffOperator}, and the function
\begin{equation}
\label{eq:def_f}
f_{h,x;t}(w):= e^{-\frac{1}{3}tw^3 +xw^2 +hw},\quad w\in\complexC.
\end{equation}
\end{defn}

Using the definitions of $\rC$ and $\rD$ in \eqref{eq:def_CauchyDet} and $\eqref{eq:def_DiffOperator}$, it is straightforward to write down $\rI_n(\vec{U};\vec{V})$ explicitly. We have
\begin{equation}
\label{eq:rI_explicit}
\begin{split}
&\rI_n(\vec{U};\vec{V})\\
& = (-1)^{n}e^{-\frac{2}{3}t+2h}
\frac{\prod_{1\le i<j\le n}(u_i-u_j)^2(v_i-v_j)^2}{\prod_{i=1}^n\prod_{j=1}^n (u_i-v_j)^2}  \cdot \prod_{i=1}^n \frac{(1+v_i)(1-u_i)f_{h,x;t}(u_i)}{(1+u_i)(1-v_i)f_{h,x;t}(v_i)}  \cdot \rH(\vec{U}\sqcup(1);\vec{V}\sqcup(-1)),
\end{split}
\end{equation}
where the function $\rH$ is defined in \eqref{eq:def_rH}, more explicitly,
\begin{equation}
\begin{split}
    &\rH(\vec{U}\sqcup(1);\vec{V}\sqcup(-1))\\
    &=\frac{1}{12}\left(2+\sum_{i=1}^n (u_i-v_i)\right)^4 +\frac14\left(\sum_{i=1}^n(u_i^2-v_i^2)\right)^2 -\frac13\left(2+\sum_{i=1}^n (u_i-v_i)\right)\left(2+\sum_{i=1}^n (u_i^3-v_i^3)\right).
    \end{split}
\end{equation}

Note the angles of the contours $\Gamma_\rL$ and $\Gamma_\rR$. We see that the function $f_{h,x;t}(w)$ decays super-exponentially fast when $w\in\Gamma_\rL$ goes to infinity, and grows super-exponentially fast when $w\in\Gamma_\rR$ goes to infinity. Therefore, it is straightforward to see that the integral of $\rI_n(\vec{U};\vec{V})$ is well defined, and the summation in \eqref{eq:def_rp} is uniformly bounded. We also note that $|\hat\rI_n(\vec{U};\vec{V})|\le \text{constant}\cdot|\rI_n(\vec{U};\vec{V})|$ by the choice of the contours. Therefore the integral of $\hat\rI_n(\vec{U};\vec{V})$ is well defined and the summation in \eqref{eq:def_hat_rp} is uniformly bounded as well. It further verifies the two formulas \eqref{eq:def_hat_rp01} and \eqref{eq:def_hat_rp} are equivalent. The absolute convergences of similar expressions have been discussed in \cite{Liu22c,Liu-Zhang25}, hence we omit the details here and refer the readers to the discussions after Definition 2.2 in \cite{Liu-Zhang25}.

\begin{figure}
\includegraphics[scale=0.28]{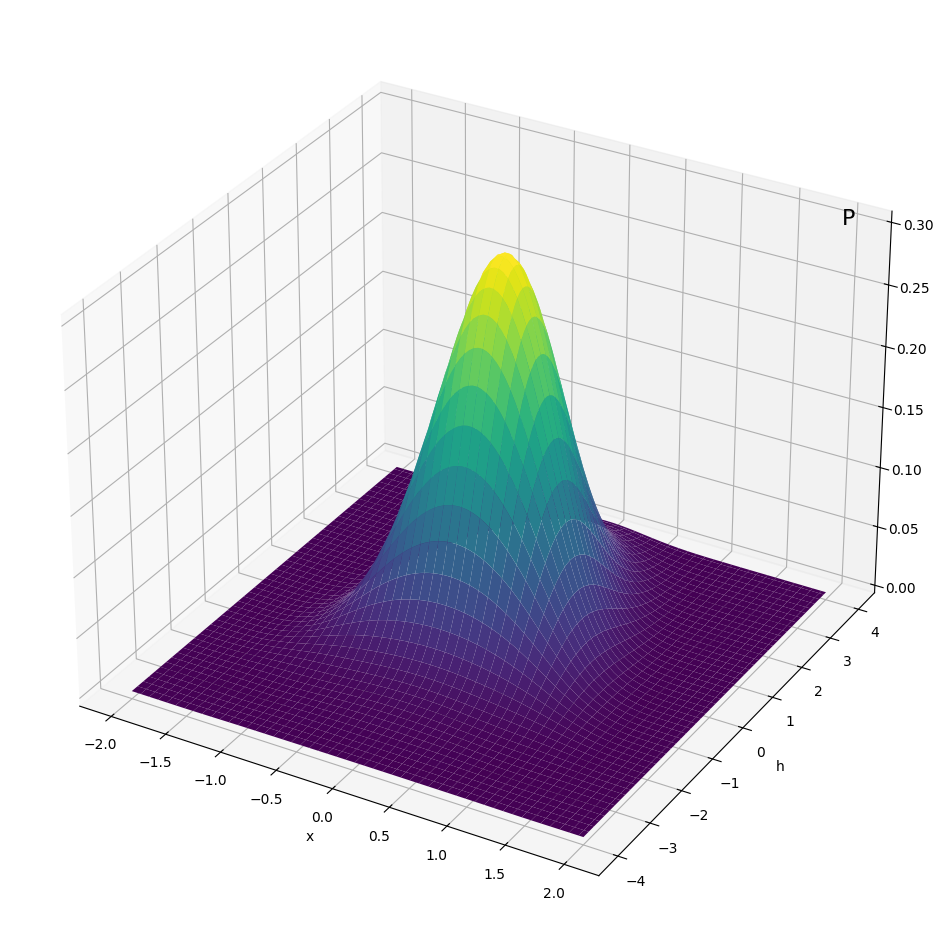}
\includegraphics[scale=0.28]{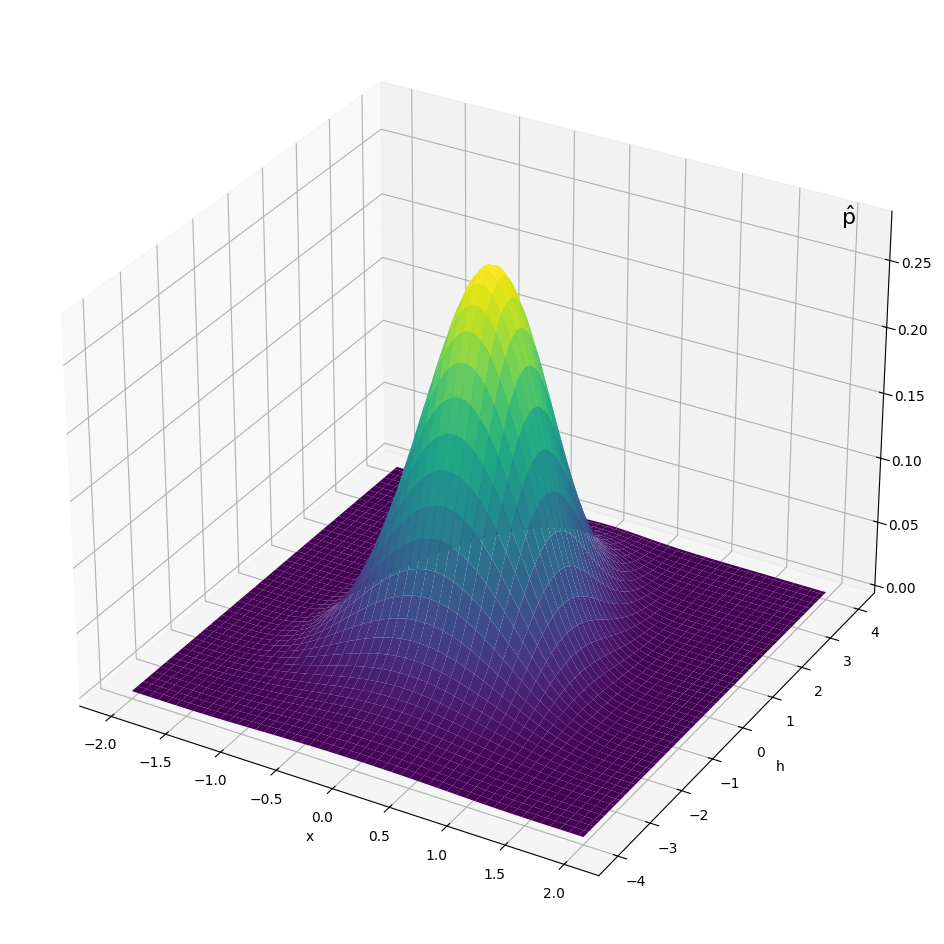}
\caption{The two figures above are approximations of the joint density functions $\rp(h,x;t=1)$ and $\hat\rp(h,x;t=1)$, respectively. Due to the complexity of their formulas, only the first term was numerically computed, with the remaining terms omitted.}
\label{fig:density_functions}
\end{figure}

\subsection{Relations to the KPZ fixed point with the upper tail conditioning}

\label{sec:KPZ}

 The KPZ fixed point $\hKPZ(x,t)$, $(x,t)\in\realR\times\realR_{>0}$, which was first constructed by \cite{Matetski-Quastel-Remenik21}, is a universal limit of the space-time height function of the Kardar-Parisi-Zhang universality class \cite{Nica-Quastel-Remenik20,Quastel-Sarkar23,Virag20,Aggarwal-Corwin-Hegde24}. It has the following connection to the directed landscape (see \cite{Nica-Quastel-Remenik20})
\begin{equation}
    \label{eq:relation_KPZ_DL}
\hKPZ (x,t) \stackrel{d}{=} \max_{x'\in\realR} \left\{ \hKPZ (x', 0) + \dL(x',0;x,t) \right\},\quad x\in\realR, \quad t>0,
\end{equation}
where $\hKPZ (x, 0)$ is the initial data of the KPZ fixed point. We will only consider a special initial condition, the narrow-wedge initial condition $\hKPZ(x,0) = -\infty \identity_{x\ne 0}$ in this paper. Thus, we simply use the same notation $\hKPZ(x,t)$ for the KPZ fixed point with the narrow-wedge initial condition throughout the rest of the paper. The equation \eqref{eq:relation_KPZ_DL} becomes 
\begin{equation}
    \label{eq:relation_KPZ_DL2}
    \hKPZ(x,t) \stackrel{d}{=} \dL(0,0;x,t),\quad x\in\realR, \quad t>0.
\end{equation}

Although the geodesics in the directed landscape are determined by the full field $\dL(x,s;y,t)$ by \eqref{eq:def_geodesic}, the relationship between the KPZ fixed point and the directed landscape, given by \eqref{eq:relation_KPZ_DL2}, suggests that the geodesic information may be intricately intertwined with that of the KPZ fixed point in a subtle, non-explicit way. Below, we explain that under the upper tail conditioning, the geodesic in the directed landscape is heuristically the same as the argmax process of the KPZ fixed point when $t=O(1)$.

In \cite{Liu-Wang24}, the authors considered the KPZ fixed point with the upper tail conditioning and found that 
\begin{multline}
\label{eq:BBlimits}
\law\left(\left\{L^{-1/4} \left( \hKPZ\left(\frac{x}{2L^{1/4}}, t\right) -tL\right) \right\}_{(x,t)\in\realR\times(0,1)} \Big|\ \hKPZ(0,1)=L\right)\\
\fddto \left\{\bbB_2(t) - |\bbB_1(t) -x|\right\}_{(x,t)\in\realR\times(0,1)}
\end{multline}
as $L\to\infty$, where $\bbB_1$ and $\bbB_2$ are two independent standard Brownian bridges, and the $\fddto$ denotes the convergence of finite dimensional distributions. Note that $\hKPZ(0,1)= \dL(1)$. One can easily observe that $\bbB_1$ and $\bbB_2$ are the argmax and max processes of the limiting field \eqref{eq:BBlimits}. Therefore, we heuristically have
\begin{equation}
\left\{\left(\argmax_{x} \hKPZ\left(\frac{x}{2L^{1/4}}, t\right),L^{-1/4}\max_x \left( \hKPZ\left(\frac{x}{2L^{1/4}}, t\right) -tL\right) \right)\Big|\ \hKPZ(0,1)=L\right\}_{t\in(0,1)}
\end{equation}
converges to $\{\bbB_1(t),\bbB_2(t)\}_{t\in(0,1)}$, assuming the left-hand side is well defined.

 Recall that the limiting distribution of $L^{-1/4}(\dL(t)-tL)$ is Gaussian \eqref{eq:GaussianLimit_OnePoint}, and it matches the distribution of $\bbB_2(t)$ for any fixed $t$. Moreover, the limiting distribution of $L^{-1/4}\left(\max_x \hKPZ(x/(2L^{1/4}),t) -tL\right)$ is always greater than or equal to $L^{-1/4}(\dL(t)-tL)$. Therefore, the two random variables should asymptotically match. Based on these observations, \cite[Conjectures 1.10, 1.11]{Liu-Wang24} conjectured that, with the upper tail conditioning, the geodesic $\pi(t)$ and its length $\dL(t)$ are the same as the argmax and the max processes of the KPZ fixed point after appropriate scaling. More explicitly, they conjectured that
\begin{equation}
\label{eq:conjecture}
    \begin{split}
    &\left\{2L^{1/4}\pi (t), L^{-1/4}(\dL(t)-tL)\right\}_{t\in(0,1)},\\
    & \left\{\argmax_{x}  \hKPZ\left(\frac{x}{2L^{1/4}}, t\right),L^{-1/4}\max_x \left( \hKPZ\left(\frac{x}{2L^{1/4}}, t\right) -tL\right)\right\}_{t\in(0,1)}
    \end{split}
\end{equation}
have the same limiting law conditioning on $\dL(1)=\hKPZ(0,1)=L$ as $L\to\infty$, and the limit is given by $\left\{\bbB_1(t),\bbB_2(t)\right\}_{t\in (0,1)}$.

\bigskip

Note that the discussions above are about the limiting behaviors of the geodesic in the directed landscape and the KPZ fixed point when the time parameter $t$ is of order $O(1)$ under the upper tail conditioning. Very recently, it was shown in \cite{Liu-Zhang25} that, with the same upper tail conditioning, there is a critical scaling near the point $(0,1)$ with which the conditional KPZ fixed point converges to a new space time random field that interpolates a Brownian-like field and the KPZ fixed point $\hKPZ$. This new field is called the upper tail field. More explicitly, they showed that 
\begin{equation}
    \label{eq:UpperTailFieldConvergence}
    \law\left( \left\{\sqrt{L}( \hKPZ(x L^{-1}, 1+ tL^{-3/2} ) -L)\right\}_{(x,t)\in\realR^2} \big|\ \hKPZ(0,1)\ge L \right)\fddto \law\left(\left\{ \LUTfield (x,t)\right\}_{(x,t)\in\realR^2}\right)
\end{equation}
and
\begin{equation}
    \label{eq:UpperTailFieldConvergence1}
    \law\left( \left\{\sqrt{L}(\hKPZ(x L^{-1}, 1+ tL^{-3/2} ) -L)\right\}_{(x,t)\in\realR^2} \big|\ \hKPZ(0,1)=L \right)\fddto \law\left(\left\{ \UTfield (x,t)\right\}_{(x,t)\in\realR^2}\right),
\end{equation}
where $\LUTfield$ and $\UTfield$ are called the upper tail fields of the KPZ fixed point. We call $\LUTfield$ the lifted upper tail field and $\UTfield$ the grounded upper tail field. They satisfy the following relation 
\begin{equation}
\label{eq:relation_two_UTfields}
    \LUTfield(x,t) = \UTfield(x,t) + E,
\end{equation}
where $E=\LUTfield(0,0)$ is an exponential random variable with parameter $2$ which is independent of $\UTfield$. Note that $\UTfield(0,0)=0$. The grounded upper tail field also satisfies that $\UTfield (x,0)$ has the same distribution as $\bfB_{\mathrm{ts}}(2x) - 2|x|$, where  $\bfB_{\mathrm{ts}}$ denotes a two-sided Brownian motion with $\bfB_{\mathrm{ts}}(0)=0$.

Recall our main results in Theorem \ref{thm:main} about the limiting geodesic distributions within the same scaling window. By comparing it with the distributions of the lifted upper tail field obtained in \cite{Liu-Zhang25}, we find the following surprising connection.
\begin{thm}
    \label{thm:connection_rp_UTfield}
    For any $h, x\in\realR$ and $t>0$, we have 
    \begin{equation}
        \rp(h,x;t) = -\left(\frac{1}{12}\partial^4_h +\frac{1}{4}\partial^2_x + \partial_t\partial_h\right) \prob\left(\LUTfield(x,-t) \ge -h \right).
    \end{equation}
\end{thm}
\begin{proof}[Proof of Theorem \ref{thm:connection_rp_UTfield}]
The proof is straightforward. The joint tail probabilities of $\LUTfield$ was explicitly formulated in \cite{Liu-Zhang25}. In particular,  $\prob\left(\LUTfield(x,-t) \ge -h \right)$ has exactly the same formula as for $\rp(h,-x;t)$ in \eqref{eq:def_rp}, except that there is no differentiation operator $\rD$ in the integrand \eqref{eq:def_rI}, and a negative sign in front. See the Definition 2.4 and the formulas (2.67) and (2.68) in \cite{Liu-Zhang25}. The Proposition 2.7 in \cite{Liu-Zhang25} claimed that any differentiation of the tail probability functions of the field $\LUTfield$ can be taken inside the summation and integral. Therefore the differentiation $\left(\frac{1}{12}\partial^4_h +\frac{1}{4}\partial^2_x + \partial_t\partial_h\right)$ is converted to the operator $\rD$ in \eqref{eq:def_rI}. Also note that $\rp(h,x;t)=\rp(h,-x;t)$ by Proposition \ref{prop:properties_p}. The desired identity follows immediately.
\end{proof}

\begin{rmk}
\label{rmk:rel2ptKPz}
A similar argument implies $p(\beta,\beta',\alpha;\tau)$, the joint density function of $\dL(\tau), \dL(1)-\dL(\tau)$, and $\pi(\tau)$ defined in Definition \ref{def:density} in the next section, also has a connection to the two-time joint distribution function of the KPZ fixed point with the narrow-wedge initial condition. One can write
\begin{equation}
\begin{split}
    p(\beta,\beta',\alpha;\tau) &= \oint_0\frac{\rd z}{2\pi\rmi (1-z)^2} \bfT(z;\beta,\beta',\alpha;\tau),\\
    \prob\left(\hKPZ(\alpha_1,\tau_1)\le \beta_1, \hKPZ(\alpha_2,\tau_2)\le \beta_2\right) &= \oint_{0} \frac{\rd z}{2\pi\rmi z(1-z)} \bfD (z;\beta_1,\beta_2; (\alpha_1,\tau_1),(\alpha_2,\tau_2)),\quad \tau_2>\tau_1,
\end{split}
\end{equation}
where $\bfT$ and $\bfD$ are explicitly defined by a double sum of contour integrals, see Definition \ref{def:density} below and Theorem 2.19 in \cite{Liu22}. Then the following identity holds, by a simple comparison of their formulas,
\begin{equation}
\label{eq:differential_equation}
\begin{split}
&\bfT(z;\beta,\beta',\alpha;\tau) \\
&= 2\left.\left[\left(\frac{1}{12}\partial_{\beta_1}^4 +\frac14 \partial_{\alpha_1}^2 +\partial_{\tau_1}\partial_{\beta_1}\right)\bfD (\beta_1,\beta_2;z; (\alpha_1,\tau_1),(\alpha_2,\tau_2))\right]\right|_{\substack{\beta_1=\beta',\alpha_1=-\alpha,\tau_1=1-\tau,\\ \beta_2=\beta+\beta',\alpha_2=0, \tau_2=1}}.
\end{split}
\end{equation}
It was known that the function $\bfD$ satisfies some differential equations \cite{Baik-Prokhorov-Silva23}. It might be an interesting question to investigate the quantity in \eqref{eq:differential_equation} from the perspective of differential equations.
\end{rmk}

The identity in Theorem \ref{thm:connection_rp_UTfield} indicates that with the upper tail conditioning, there is a deeper connection between the geodesic and the KPZ fixed point  that we do not fully understand. It would also be an interesting question to give a probabilistic proof of Theorem \ref{thm:connection_rp_UTfield}, but this is beyond the scope of this paper. Another interesting question is whether there exists a relation analogous to \cite[Conjectures 1.10, 1.11]{Liu-Wang24} between the conditional geodesic and the upper tail field, both of which are conditional limits of the directed landscape or KPZ fixed point in the same scaling window. 

\subsection{Main idea of the proof and structure of the paper}

The proofs of the results in this paper involve asymptotic analysis for series expansions of contour integrals, and probabilistic arguments for the KPZ fixed point and directed landscape. Similar asymptotic analysis for this type of expansions has been considered in earlier papers \cite{Liu22c,Liu22b,Liu-Wang24,Liu-Zhang25}, and our analysis in this paper adapts ideas from these papers in our settings. With that being said, our results are novel and some results are surprising, such as Theorem \ref{thm:connection_rp_UTfield}. Moreover, the most challenging part of the proofs in this paper actually comes from that of Proposition \ref{prop:joint_density}, i.e., the claim that $\rp$ and $\hat\rp$ are probability density functions. We introduce some nontrivial probabilistic arguments to bound the geodesic location and geodesic length when the time window is as small as $O(L^{-3/2})$. We leverage many properties of the directed landscape and the KPZ fixed point, while also applying the latest results from the recently defined upper tail field of the KPZ fixed point. We believe that the ideas of these probabilistic arguments could potentially be useful in proving similar properties of the geodesic in the directed landscape.

Below is the structure of the rest of the paper.

Section \ref{sec:proof_theorem} is the proof of Theorem \ref{thm:main}. We conduct a steepest descent analysis of the joint density function, derived in \cite{Liu22c}, for the geodesic location and geodesic lengths, incorporating the scaling considered in this paper.

In Section \ref{sec:proof_density}, we prove Proposition \ref{prop:joint_density} using probabilistic arguments. The main results in this section are Lemmas \ref{lm:bound_geodesic_location} and \ref{lm:bound_geodesic_length}, which we prove in subsections \ref{sec:geodecis_location} and \ref{sec:geodesic_length} respectively.

Finally, in Section \ref{sec:proof_proposition2}, we provide the proof of Proposition \ref{prop:properties_p}. The two more involved properties about the large time limit and the large $x$ or $h$ limit of the joint density functions are proved in two subsections \ref{sec:large_time_limit} and \ref{sec:right_tail} respectively.

\section*{Acknowledgements}

ZL is partly supported by NSF DMS-2246683. We thank the anonymous referee for their valuable suggestions and comments.

\section{Proof of Theorem \ref{thm:main}}
\label{sec:proof_theorem}

Recall that $\pi(\tau)$ denotes the location of the geodesic in the directed landscape $\dL$ from $(0,0)$ to $(0,1)$ at time $0\le\tau\le 1$. Moreover, $\dL(\tau) = \dL(0,0;\pi(\tau),\tau)$ and $\dL(1)-\dL(\tau) =\dL(\pi(\tau),\tau;0,1)$ for $0<\tau< 1$. The proof of Theorem \ref{thm:main} relies on an exact formula of the joint density of  $\dL(\tau)$, $\dL(1)-\dL(\tau)$, and $\pi(\tau)$ obtained in \cite{Liu22c}. Then we apply the asymptotic analysis after appropriately choosing the parameters.

There are four statements \eqref{eq:newmain0}, \eqref{eq:newmain}, \eqref{eq:main0}, and \eqref{eq:main}. All of them can be deduced from the asymptotic analysis (with two different settings of parameters) of some explicit expressions using the joint density of $\dL(\tau)$, $\dL(1)-\dL(\tau)$, and $\pi(\tau)$. Due to the similarity of the analysis, we only show the proof of \eqref{eq:main0} and heuristically derive the other three statements, which could also be proved rigorously in the same way as for \eqref{eq:main0}. 

Now we explain the heuristic argument for \eqref{eq:newmain0}, \eqref{eq:newmain}, and \eqref{eq:main} using \eqref{eq:main0}. The equivalence between \eqref{eq:main} and \eqref{eq:main0} is due to the flip symmetry of the directed landscape (see \cite[Lemma 10.2]{Dauvergne-Ortmann-Virag22}). More explicitly, the joint law of $\pi(t),\dL(t)$ and $\dL(1)$ is the same as that of $\pi(1-t),\dL(1)-\dL(1-t)$ and $\dL(1)$, hence the left-hand sides of \eqref{eq:main} and \eqref{eq:main0} are equal, so are their limits. 

In order to see \eqref{eq:newmain0} and \eqref{eq:newmain}, we need the following asymptotics (see \cite[Equation (25)]{Baik-Buckingham-DiFranco08} and \cite[Equation (3.4)]{Liu22b})
\begin{equation} \label{eq:GUE_asymp}
    \fGUE(L) = \frac{e^{-\frac{4}{3}L^{3/2}}}{8 \pi L}(1 + O(L^{-3/2})),\quad 1-\FGUE(L) = \frac{e^{-\frac{4}{3}L^{3/2}}}{16 \pi L^{3/2}}(1 + O(L^{-3/2})),
\end{equation}
where $\fGUE$ and $\FGUE$ are the density function and the cumulative distribution function of $\dL(1)$.
Therefore, assuming  \eqref{eq:main0}, we heuristically have 
\begin{equation}
    \begin{split}
        &\prob\left( L \pi (tL^{-3/2}) \in (x_1,x_2), L^{1/2}\dL(tL^{-3/2}) \in (h_1,h_2) \ \big| \ \dL(1)\ge L \right)\\
        &=  \int_0^\infty \prob\left( L \pi (tL^{-3/2}) \in (x_1,x_2), L^{1/2}\dL(tL^{-3/2}) \in (h_1,h_2) \ \big| \ \dL(1)=L+yL^{-1/2} \right)\cdot\frac{\fGUE(L + yL^{-1/2})}{L^{1/2}(1-\FGUE(L))}\rd y\\
        &\to \int_0^\infty \int_{h_1}^{h_2} \int_{x_1}^{x_2} \rp(h,x;t) \rd x \rd h \cdot 2e^{-2y} \rd y
    \end{split}
\end{equation}
as $L\to\infty$, where we took the limit inside the integral and used the results \eqref{eq:main0} and \eqref{eq:GUE_asymp}. This gives \eqref{eq:newmain0}. Similarly,
\begin{equation}
    \begin{split}
        &\prob\left( L \pi (1-tL^{-3/2}) \in (x_1,x_2), L^{1/2}(L-\dL(1-tL^{-3/2})) \in (h_1,h_2) \ \big| \ \dL(1)\ge L \right)\\
        &=\int_0^\infty \prob\left( L \pi (1-tL^{-3/2}) \in (x_1,x_2), L^{1/2}(L-\dL(1-tL^{-3/2})) \in (h_1,h_2) \ \big| \ \dL(1)= L+yL^{-1/2} \right)\\
        &\quad \cdot\frac{\fGUE(L + yL^{-1/2})}{L^{1/2}(1-\FGUE(L))}\rd y\\
        &\to \int_0^\infty \int_{h_1+y}^{h_2+y} \int_{x_1}^{x_2} \rp(h,x;t) \rd x \rd h \cdot 2e^{-2y} \rd y= \int_{h_1}^{h_2} \int_{x_1}^{x_2} \int_0^\infty \rp(h+y,x;t)\cdot 2e^{-2y} \rd y \rd x \rd h
    \end{split}
\end{equation}
as $L\to\infty$. Recall the definition of $\hat\rp$ in \eqref{eq:def_hat_rp01}. The statement \eqref{eq:newmain} follows immediately.

\bigskip

In the rest of this section, we focus on the proof of \eqref{eq:main0}. We will first state the joint density function of  $\dL(\tau)$, $\dL(1)-\dL(\tau)$, and $\pi(\tau)$ in the first subsection. Then we apply the asymptotic analysis and obtain \eqref{eq:main0} in the second subsection.

\subsection{The joint density function of the geodesic location and lengths}

Our starting point is the following proposition which gives the exact formula of the joint density function of  $\dL(\tau)$, $\dL(1)-\dL(\tau)$, and $\pi(\tau)$.

\begin{prop}[\cite{Liu22c,Liu22b}]
\label{prop:geodesic_density}
Assume $0<\tau<1$. We have
\begin{equation}
\begin{split}
&\prob\left( \dL(\tau) \in (\betasmin,\betasmax), \dL(1)-\dL(\tau) \in (\betaemin,\betaemax), \pi(\tau) \in (\alphamin,\alphamax) \right) \\
&= \int_{\betasmin}^{\betasmax} \int_{\betaemin}^{\betaemax} \int_{\alphamin}^{\alphamax} \density \left( \beta , \beta'  ,\alpha ;\tau\right) \rd \alpha \rd \beta' \rd \beta
\end{split}
\end{equation}
for any $\betasmin,\betasmax,\betaemin,\betaemax,\alphamin,\alphamax \in \realR$ satisfying $\betasmin<\betasmax,\betaemin<\betaemax$ and $\alphamin<\alphamax$, where the joint density function $\density$ is defined below. See Definition \ref{def:density}.
\end{prop}

We remark that we slightly modified the notations in \cite[Proposition 2.1]{Liu22b}. The density function (2.1) in \cite[Proposition 2.1]{Liu22b} corresponds to $\density(\ell_1,\ell_2,-x;s)= \density(\ell_1,\ell_2,x;s)$ in our notations since $\density \left( \beta , \beta'  ,\alpha ;\tau\right) $ is symmetric in $\alpha$ (from the definition of the geodesic, or from the formula directly).

In order to write down the definition of the joint density function $\density$, we need some notations. We first introduce six contours $C_{\rL}^\inn, C_{\rL},C_{\rL}^\out,C_{\rR}^\out, C_{\rR}$, and $C_{\rR}^\inn$. The contours $C_{\rL}^\inn, C_{\rL},C_{\rL}^\out$, from left to right, are three parallel contours that go from $\infty e^{-2\pi\rmi/3}$ to $e^{2\pi\rmi/3}$. The other three contours are their symmetrization about the imaginary axis, i.e., they all go from  $\infty e^{-\pi\rmi/3}$ to $e^{\pi\rmi/3}$ and $C_{\rR}^\out=-C_{\rL}^\out, C_\rR=-C_\rL, C_{\rR}^\inn=-C_{rL}^\inn$. See Figure \ref{fig:C_Contours} for an illustration.

\begin{figure}
\centering
\begin{tikzpicture}[
    x=1cm, y=0.6cm 
    ]
    \def\R{4} 
    \def\angle{120} 
    
    \path 
        ({\R * cos(\angle)}, {-\R * sin(\angle)}) coordinate (A) 
        ({\R * cos(\angle)}, {\R * sin(\angle)}) coordinate (C); 
    \path 
        ({\R * cos(60)}, {-\R * sin(60)}) coordinate (D) 
        ({\R * cos(60)}, {\R * sin(60)}) coordinate (F); 
        
    \draw[->] ({\R * cos(\angle) - 1}, 0) -- ({-\R * cos(\angle) + 1}, 0) node[below] {Re};
    \draw[->] (0, {-\R * sin(\angle) - 1}) -- (0, {\R * sin(\angle) + 1}) node[left] {Im};
    
    \draw[blue, thick, ->] 
        (A) 
        .. controls (-0.8, 0) .. (C) 
        node[midway, right, black] {};
    \draw[blue, thick, ->] 
        ({\R * cos(\angle)-0.8}, {-\R * sin(\angle)}) 
        .. controls (-0.8-0.8, 0) .. ({\R * cos(\angle)-0.8}, {\R * sin(\angle)}) 
        node[midway, right, black] {};
    \draw[blue, thick, ->] 
        ({\R * cos(\angle)+ 0.8}, {-\R * sin(\angle) }) 
        .. controls (-0.8+ 0.8, 0) .. ({\R * cos(\angle)+ 0.8}, {\R * sin(\angle) }) 
        node[midway, right, black] {};
    
    \draw[red, thick, ->] 
        (D) 
        .. controls (0.8, 0) .. (F) 
        node[midway, right, black] {};
    \draw[red, thick, ->] 
        ({\R * cos(60)- 0.8}, {-\R * sin(60) }) 
        .. controls (0.8-0.8, 0) .. ({\R * cos(60)- 0.8}, {\R * sin(60) }) 
        node[midway, right, black] {};
    \draw[red, thick, ->] 
        ({\R * cos(60)+ 0.8}, {-\R * sin(60) }) 
        .. controls (0.8+ 0.8, 0) .. ({\R * cos(60)+ 0.8}, {\R * sin(60)}) 
        node[midway, right, black] {};
    
    \node[below left] at (A) {$\infty e^{-2\pi \mathrm{i}/3}$};
    \node[above left] at (C) {$\infty e^{2\pi \mathrm{i}/3}$};
    \node[above left] at (-2.5,1.9) {$C_\rL^\inn$}; 
    \node[above left] at (-1.5,1.9) {$C_\rL$};
    \node[above left] at (0.1,1.9) {$C_\rL^\out$};
    \node[below right] at (D) {$\infty e^{-\pi \mathrm{i}/3}$};
    \node[above left] at (2.4,1.9) {$C_\rR^\inn$}; 
    \node[above left] at (1.6,1.9) {$C_\rR$};
    \node[above left] at (0.95,1.9) {$C_\rR^\out$};
    \node[above right] at (F) {$\infty e^{\pi \mathrm{i}/3}$};
\end{tikzpicture}
\caption{Illustration of the contours $C_{\rL}^\inn, C_{\rL},C_{\rL}^\out,C_{\rR}^\out, C_{\rR},$, and $C_{\rR}^\inn$. Each contour on the left half plane goes from $\infty e^{-2\pi/3}$ to $\infty e^{2\pi/3}$, and each contour on the right half plane goes from $\infty e^{-\pi/3}$ to $\infty e^{\pi/3}$. Following the convention in \cite{Liu22c}, we use the superscripts ``in'' and ``out'' to denote the contour positions relative to the reference point $-\infty$ or $\infty$.}
\label{fig:C_Contours}
\end{figure}
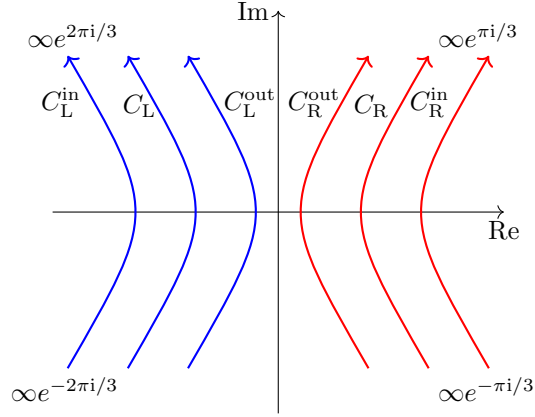

Now we are ready to define the following density function $\density$.

\begin{defn}
\label{def:density}
The density function $\density$ is given by the following formula
\begin{equation}
    \label{eq:def_density_prelimit}
	\begin{split}
	\density (\beta ,\beta',\alpha;\tau) = \oint_0 \frac{\rd z}{2\pi \rmi (1-z)^2} \sum_{n, n'\ge 1} \frac{1}{(n!)^2 (n'!)^2} T_{n,n'} (z;\beta,\beta',\alpha;\tau),
\end{split}
\end{equation}
where 
\begin{equation}
\label{eq:def_T}
	\begin{split}
		& T_{n,n'} (z;\beta,\beta',\alpha;\tau)\\
		& = 2\prod_{i=1}^{n} \left( \frac{1}{1-z} \int_{C_{\rL,\inn}} \frac{\rd \xi_{i}}{2\pi\rmi} -\frac{z}{1-z} \int_{C_{\rL,\out}} \frac{\rd \xi_{i}}{2\pi\rmi}\right)\left( \frac{1}{1-z} \int_{C_{\rR,\inn}} \frac{\rd \eta_{i}}{2\pi\rmi} -\frac{z}{1-z} \int_{C_{\rR,\out}} \frac{\rd \eta_{i}}{2\pi\rmi}\right)\\
		& \quad \prod_{i'=1}^{n'} \int_{C_\rL} \frac{\rd \xi'_{i'}}{2\pi\rmi}\int_{C_\rR} \frac{\rd \eta'_{i'}}{2\pi\rmi}
		\left(1-z\right)^{n'} \left(1- \frac{1}{z}\right)^{n} \cdot 
		\rH\left(\vec{\xi}\sqcup\vec{\eta'};\vec{\eta}\sqcup \vec{\xi'}\right) \\
		&\quad \cdot \prod_{i=1}^{n}\frac{f_{\beta,\alpha;\tau}(\xi_i)}{f_{\beta,\alpha;\tau}(\eta_{i})} \cdot \prod_{i'=1}^{n'}\frac{f_{\beta',-\alpha;1-\tau}(\xi'_{i'})}{f_{\beta',-\alpha;1-\tau}(\eta'_{i'})} \cdot 
        \rC(\Vec{\eta'};\Vec{\xi'})\cdot \rC(\Vec{\xi'}\sqcup\Vec{\eta}; \Vec{\eta'}\sqcup \Vec{\xi})\cdot \rC(\Vec{\xi};\Vec{\eta}).
	\end{split}
\end{equation}
The vectors $\Vec{\xi}=(\xi_1,\ldots,\xi_{n}), \Vec{\eta}=(\eta_1,\ldots,\eta_n)$, and $\Vec{\xi'}=(\xi'_1,\ldots,\xi'_{n'}), \Vec{\eta'} =(\eta'_1,\ldots,\eta'_{n'})$. $\rC$ is the Cauchy determinant defined by \eqref{eq:def_CauchyDet}, and $\rH$ is the function defined in \eqref{eq:def_rH}. Finally, the functions $f_{a,b;c}$ are defined by \eqref{eq:def_f}; More explicitly,
\begin{equation}
f_{\beta,\alpha;\tau}(\zeta) = e^{-\frac{\tau}{3}\zeta^3 +\alpha \zeta^2 +\beta\zeta},\quad f_{\beta',-\alpha;1-\tau}(\zeta) = e^{-\frac{1-\tau}{3}\zeta^3 -\alpha\zeta^2 +\beta'\zeta}
\end{equation}
for all $\zeta\in\complexC$.
\end{defn}

\medskip

As mentioned in Remark \ref{rmk:rel2ptKPz}, the density function $p$ has a very similar structure with the two-point probability distribution function of the KPZ fixed point. Below we present a version of the two-point tail probability function of the KPZ fixed point that could be found in \cite[Proposition 3.1]{Liu-Zhang25} with a special choice of parameters. We use the tail probability function instead of the probability distribution function here, because the asymptotics of the tail probability function with the same scalings has been considered in \cite{Liu-Zhang25} and we will be able to modify their results directly instead of redoing all the asymptotics. We also  modify the notations of the parameters to match that in the function $p$.

\begin{lm}[\cite{Liu22},\cite{Liu-Zhang25}]
    We have the following two-point tail probability function of the KPZ fixed point with the narrow-wedge initial condition 
    \begin{equation}
        \label{eq:tail_KPZ_2point}
        \prob\left(\hKPZ(-\alpha,1-\tau)\ge \beta', \hKPZ(0,1)\ge \beta+\beta'\right) =\oint_{>1} \frac{\rd z}{2\pi\rmi z(1-z)} \sum_{n,n'\ge 1} \frac{1}{(n!)^2(n'!)^2} D_{n,n'}(z;\beta,\beta',\alpha;\tau),
    \end{equation}
    where 
    \begin{equation}
    \label{eq:def_D}
        \begin{split}
		& D_{n,n'} (z;\beta,\beta',\alpha;\tau)\\
		& = \prod_{i=1}^{n} \left( \frac{1}{1-z} \int_{C_{\rL,\inn}} \frac{\rd \xi_{i}}{2\pi\rmi} -\frac{z}{1-z} \int_{C_{\rL,\out}} \frac{\rd \xi_{i}}{2\pi\rmi}\right)\left( \frac{1}{1-z} \int_{C_{\rR,\inn}} \frac{\rd \eta_{i}}{2\pi\rmi} -\frac{z}{1-z} \int_{C_{\rR,\out}} \frac{\rd \eta_{i}}{2\pi\rmi}\right)\\
		& \quad \prod_{i'=1}^{n'} \int_{C_\rL} \frac{\rd \xi'_{i'}}{2\pi\rmi}\int_{C_\rR} \frac{\rd \eta'_{i'}}{2\pi\rmi}
		\left(1-z\right)^{n'} \left(1- \frac{1}{z}\right)^{n} \cdot  \prod_{i=1}^{n}\frac{f_{\beta,\alpha;\tau}(\xi_i)}{f_{\beta,\alpha;\tau}(\eta_{i})} \cdot \prod_{i'=1}^{n'}\frac{f_{\beta',-\alpha;1-\tau}(\xi'_{i'})}{f_{\beta',-\alpha;1-\tau}(\eta'_{i'})} \\
        &\quad \cdot 
        \rC(\Vec{\eta'};\Vec{\xi'})\cdot \rC(\Vec{\xi'}\sqcup\Vec{\eta}; \Vec{\eta'}\sqcup \Vec{\xi})\cdot \rC(\Vec{\xi};\Vec{\eta}).
	\end{split}
    \end{equation}
    Here the integration contours and the functions $\rC$ and $f$ are the same as in Definition \ref{def:density}. 
\end{lm}

Note that the formula of $D_{n,n'}$ is almost identical to that of $T_{n,n'}$ in \eqref{eq:def_T}, except for the factors $2$ and $\rH(\vec{\xi}\sqcup\vec{\eta'};\vec{\eta}\sqcup\vec{\xi'})$ in \eqref{eq:def_T}. Moreover, the outside $z$ integral  in \eqref{eq:tail_KPZ_2point} is different from that in \eqref{eq:def_density_prelimit}.

\subsection{Asymptotic analysis and the proof of \eqref{eq:main0}}

We will consider the asymptotics of $\density(\beta,\beta',\alpha;\tau)$ when the parameters satisfy 
\begin{equation} 
    \label{eq:scale_beta_alpha_tau}
    \beta =   h L^{-1/2}, \quad \beta' = L -h L^{-1/2}, \quad \alpha =  x L^{-1}, \quad \tau =  t L^{-3/2},
\end{equation}
where $(h,x)$ stays in a compact subset of $\realR^2$ and $t>0$ is fixed.

The asymptotic analysis of the multipoint (including the two-point) tail probability function of the KPZ fixed point in this scaling window was considered in \cite{Liu-Zhang25}. Due to the similarity of the structure of $\density(\beta,\beta',\alpha;\tau)$ and the function $\prob\left(\hKPZ(-\alpha,1-\tau)\ge \beta', \hKPZ(0,1)\ge \beta+\beta'\right)$, it is quite straightforward to modify the results in \cite{Liu-Zhang25} to obtain the asymptotics of $\density(\beta,\beta',\alpha;\tau)$. Below we explain how the asymptotic analysis works and show how to derive the asymptotics of $\density(\beta,\beta',\alpha;\tau)$. We will not present all the details since they were already written in \cite{Liu-Zhang25}. Instead, we focus on the idea of the analysis and how to obtain our results.

It turns out that with the setting of the parameters in \eqref{eq:scale_beta_alpha_tau}, the main contribution of the summation of terms $D_{n,n'}$ comes from the case $n'=1$, and the variables need to be rescaled as 
\begin{equation}
\label{eq:xi_eta_rescale}
    \xi_i = \sqrt{L} u_i,\quad \eta_i =\sqrt{L} v_i, \quad \xi'_1\approx -\sqrt{L},\quad \eta'_1\approx \sqrt{L},
\end{equation} 
where $u_i\in\Gamma_\rL$ if $\xi_i\in  C_{\rL,\out}$, or $u_i\in -1+\Gamma_\rL$ if $\xi_i\in  C_{\rL,\inn}$. Similarly, $v_i\in\Gamma_{\rR}$ if $\eta_i\in C_{\rR,\out}$, or $v_i\in 1+\Gamma_{\rR}$ if $\eta_i\in C_{\rR,\inn}$. Here $\Gamma_\rL$ and $\Gamma_\rR$ are contours defined in Section \ref{sec:Def_rp}, see Figure \ref{fig:Gamma_Contours} for an illustration. With this rescaling, the integrand in $D_{n,n'}$ decays super-exponentially fast as the variables go to infinity along the new integration contours, and the dominated convergence theorem applies for \eqref{eq:def_D}. The following results were proved in \cite[Lemma 3.3, Lemma 3.4]{Liu-Zhang25}.

\begin{lm}[\cite{Liu-Zhang25}]
    \label{lm:asymptotics_D}
    Assume \eqref{eq:scale_beta_alpha_tau} where $(h,x)$ stays in a compact subset of $\realR^2$ and $t>0$ is fixed. Then 
    \begin{equation}
        \begin{split}
            &\lim_{L\to\infty}16\pi L^{3/2}e^{\frac43L^{3/2}}D_{n,n'} (z;\beta,\beta',\alpha;\tau)\\
            &= 2(1-z)\left(1-\frac1z\right)^n \prod_{i=1}^{n} \left( \frac{1}{1-z} \int_{-1+\Gamma_\rL} \frac{\rd u_{i}}{2\pi\rmi} -\frac{z}{1-z} \int_{\Gamma_\rL} \frac{\rd u_{i}}{2\pi\rmi}\right)\left( \frac{1}{1-z} \int_{1+\Gamma_\rR} \frac{\rd v_{i}}{2\pi\rmi} -\frac{z}{1-z} \int_{\Gamma_\rR} \frac{\rd v_{i}}{2\pi\rmi}\right)\\
            &\quad \cdot \rC((-1)\sqcup \vec{V}; (1)\sqcup \vec{U})\cdot \rC(\vec{U};\vec{V}) \cdot \prod_{i=1}^n \frac{f_{h,x;t}(u_i)}{f_{h,x;t}(v_i)} \cdot \frac{f_{-h,-x;-t}(-1)}{f_{-h,-x;-t}(1)}
        \end{split}
    \end{equation}
    when $n'=1$, and 
    \begin{equation}
        \lim_{L\to\infty}16\pi L^{3/2}e^{\frac43L^{3/2}}D_{n,n'} (z;\beta,\beta',\alpha;\tau) =0
    \end{equation}
    when $n'\ge 2$. Here $\vec{U}=(u_1,\ldots,u_n)$ and $\vec{V}=(v_1,\ldots,v_n)$, and the function $f$ is defined in \eqref{eq:def_f}. Moreover,
    the following uniform bound holds
    \begin{equation}
        \left|16\pi L^{3/2}e^{\frac43L^{3/2}}D_{n,n'} (z;\beta,\beta',\alpha;\tau) \right| \le C^{n+n'}n^n(n')^{n'}\cdot \frac{(1+|z|)^n}{|1-z|^{n-n'}|z|^n},
    \end{equation}
    where $C$ does not depend on $n,n'$ or $z$.
\end{lm}

Now we consider the asymptotics of $T_{n,n'}(z;\beta,\beta',\alpha;\tau)$. Note that the only difference is the factors $2$ and  $ \rH\left(\vec{\xi}\sqcup\vec{\eta'};\vec{\eta}\sqcup \vec{\xi'}\right)$. The factor $\rH$, however, will not affect the argument in the asymptotic analysis since this factor only grows as a polynomial and the dominating factors, the factors involving the function $f$, in the integrand decay super-exponentially fast along the contour for $u_i\in \Gamma_{\rL} \cup (-1+\Gamma_{\rL})$ and $v_i\in \Gamma_{\rR} \cup (-1+\Gamma_{\rR})$. Therefore, the analysis  in \cite{Liu-Zhang25} to prove Lemma \ref{lm:asymptotics_D} goes through for the  function $T_{n,n'}(z;\beta,\beta',\alpha;\tau)$, and the contribution from the $\rH$ function gives 
\begin{equation}    \rH\left(\vec{\xi}\sqcup\vec{\eta'};\vec{\eta}\sqcup \vec{\xi'}\right) \approx L^2 \rH(\vec{U}\sqcup(1);\vec{V}\sqcup(-1)).
\end{equation}
by using \eqref{eq:xi_eta_rescale}. Therefore, we have the following result.
\begin{lm}
    \label{lm:asymptotics_T}
    Assume \eqref{eq:scale_beta_alpha_tau} where $(h,x)$ stays in a compact subset of $\realR^2$ and $t>0$ is fixed. Then 
    \begin{equation}
        \begin{split}
            &\lim_{L\to\infty}16\pi L^{-1/2}e^{\frac43L^{3/2}}T_{n,n'} (z;\beta,\beta',\alpha;\tau)\\
            &= 4(1-z)\left(1-\frac1z\right)^n \prod_{i=1}^{n} \left( \frac{1}{1-z} \int_{-1+\Gamma_\rL} \frac{\rd u_{i}}{2\pi\rmi} -\frac{z}{1-z} \int_{\Gamma_\rL} \frac{\rd u_{i}}{2\pi\rmi}\right)\left( \frac{1}{1-z} \int_{1+\Gamma_\rR} \frac{\rd v_{i}}{2\pi\rmi} -\frac{z}{1-z} \int_{\Gamma_\rR} \frac{\rd v_{i}}{2\pi\rmi}\right)\\
            &\quad \cdot \rC((-1)\sqcup \vec{V}; (1)\sqcup \vec{U})\cdot \rC(\vec{U};\vec{V}) \cdot \rH(\vec{U}\sqcup(1);\vec{V}\sqcup(-1)) \cdot \prod_{i=1}^n \frac{f_{h,x;t}(u_i)}{f_{h,x;t}(v_i)} \cdot \frac{f_{-h,-x;-t}(-1)}{f_{-h,-x;-t}(1)}
        \end{split}
    \end{equation}
    when $n'=1$, and 
    \begin{equation}
        \lim_{L\to\infty}16\pi L^{-1/2}e^{\frac43L^{3/2}}T_{n,n'} (z;\beta,\beta',\alpha;\tau) =0
    \end{equation}
    when $n'\ge 2$.   Moreover,
    the following uniform bound holds
    \begin{equation}
        \left|16\pi L^{-1/2}e^{\frac43L^{3/2}}T_{n,n'} (z;\beta,\beta',\alpha;\tau) \right| \le C^{n+n'}n^n(n')^{n'}\cdot \frac{(1+|z|)^n}{|1-z|^{n-n'}|z|^n},
    \end{equation}
    where $C$ does not depend on $n,n'$ or $z$.
\end{lm}

\medskip

Now we are ready to prove \eqref{eq:main0}. Note that 
\begin{equation*}
\begin{split}
    &\prob\left( L \pi (tL^{-3/2}) \in (x_1,x_2), L^{1/2}\dL(tL^{-3/2}) \in (h_1,h_2) \ \big| \ \dL(1)= L \right)  \\
    &= \int_{h_1}^{h_2} \int_{x_1}^{x_2}  \frac{L^{-3/2}\density(hL^{-1/2}, L - hL^{-1/2},xL^{-1};tL^{-3/2})}{\fGUE(L)} \rd x \rd h,
\end{split}
\end{equation*}
To prove \eqref{eq:main0}, it is enough to show that
\begin{equation} \label{eq:densitytorp}
    \lim_{L \to \infty}\frac{L^{-3/2}\density(hL^{-1/2}, L - hL^{-1/2},xL^{-1};tL^{-3/2})}{\fGUE(L)} = \rp(h,x;t)
\end{equation}
uniformly when $(x,h)$ stays in a compact subset of $\realR^2$ and $t>0$ is fixed. We insert the asymptotics of $\fGUE$ in \eqref{eq:GUE_asymp}, Definition \ref{def:density}, and Lemma \ref{lm:asymptotics_T} into the above formula, and obtain 
\begin{equation}
    \begin{split}
        &\lim_{L \to \infty}\frac{L^{-3/2}\density(hL^{-1/2}, L - hL^{-1/2},xL^{-1};tL^{-3/2})}{\fGUE(L)}\\
        &= \oint_0 \frac{\rd z}{2\pi\rmi(1-z)^2} \sum_{n\ge 1} \frac{1}{(n!)^2}2(1-z)\left(1-\frac1z\right)^n \\
        &\quad \prod_{i=1}^{n} \left( \frac{1}{1-z} \int_{-1+\Gamma_\rL} \frac{\rd u_{i}}{2\pi\rmi} -\frac{z}{1-z} \int_{\Gamma_\rL} \frac{\rd u_{i}}{2\pi\rmi}\right)\left( \frac{1}{1-z} \int_{1+\Gamma_\rR} \frac{\rd v_{i}}{2\pi\rmi} -\frac{z}{1-z} \int_{\Gamma_\rR} \frac{\rd v_{i}}{2\pi\rmi}\right)\\
            &\quad \cdot \rC((-1)\sqcup \vec{V}; (1)\sqcup \vec{U})\cdot \rC(\vec{U};\vec{V}) \cdot \rH(\vec{U}\sqcup(1);\vec{V}\sqcup(-1)) \cdot \prod_{i=1}^n \frac{f_{h,x;t}(u_i)}{f_{h,x;t}(v_i)} \cdot \frac{f_{-h,-x;-t}(-1)}{f_{-h,-x;-t}(1)}\\
        &=\sum_{n\ge 1} \frac{1}{(n!)^2}(-2) \prod_{i=1}^n\int_{\Gamma_\rL} \frac{\rd u_{i}}{2\pi\rmi}\int_{\Gamma_\rR} \frac{\rd v_{i}}{2\pi\rmi}\\
        &\quad \cdot \rC((-1)\sqcup \vec{V}; (1)\sqcup \vec{U})\cdot \rC(\vec{U};\vec{V}) \cdot \rH(\vec{U}\sqcup(1);\vec{V}\sqcup(-1)) \cdot \prod_{i=1}^n \frac{f_{h,x;t}(u_i)}{f_{h,x;t}(v_i)} \cdot \frac{f_{-h,-x;-t}(-1)}{f_{-h,-x;-t}(1)},
    \end{split}
\end{equation}
where we evaluated the $z$-integral in the last step by deforming the $z$-contour to infinity. A direct comparison of the right hand side of the above formula and the definition of $\rp(h,x;t)$ in Definition \ref{def:rp} implies \eqref{eq:densitytorp}. This completes the proof.

\section{Proof of Proposition \ref{prop:joint_density}}
\label{sec:proof_density}

We only need to prove that $\rp$ is a joint probability density function. The other half of the statement follows from Proposition \ref{prop:properties_p} (2). Note that both functions $\rp$ and $\hat\rp$ are nonnegative by Theorem \ref{thm:main}.

We will prove the following two lemmas in this section. We thank the referee for pointing out that Lemma \ref{lm:bound_geodesic_location} follows directly from \cite[Lemma 5.2]{Ganguly-Hegde-Zhang23}. Nevertheless, our proof offers an independent and alternative perspective on this result. To the best of our knowledge, Lemma \ref{lm:bound_geodesic_length} has not been explicitly stated elsewhere.
\begin{lm}
    \label{lm:bound_geodesic_location}
    For any $\epsilon>0$, there exists a constant $M_1>0$, such that 
    \begin{equation}
    \label{eq:uniform_bound_pi}
\prob \left(|L \pi (tL^{-3/2})| \ge  M_1 \ \big| \ \dL(1)\ge L \right) \le \epsilon
\end{equation}
for sufficiently large $L$.
\end{lm}
\begin{lm}
    \label{lm:bound_geodesic_length}
    Suppose $\epsilon>0$ and $M_1>0$ are any two fixed constants.
    \begin{enumerate}[(1)]
    \item There exists a constant $M_2$, such that 
    \begin{equation} 
    \label{eq:upperbound_geodesic_length}
\prob\left(  L^{1/2}   \dL(tL^{-3/2})  > M_2 \ \big| \ \dL(1)\ge L \right) \le \epsilon
\end{equation}
for sufficiently large $L$.
\item There exists a constant $M_3$, such that 
    \begin{equation} 
    \label{eq:lowerbound_geodesic_length}
\prob\left( L \pi (tL^{-3/2}) \in (-M_1, M_1), L^{1/2}   \dL(tL^{-3/2})  <- M_3 \ \big| \ \dL(1)\ge L \right) \le \epsilon
\end{equation}
for sufficiently large $L$.
\end{enumerate}
\end{lm}

Assuming these two lemmas, we have 
\begin{equation}
    \begin{split}
    &\prob\left(L \pi (tL^{-3/2}) \in (-M_1,M_1), L^{1/2}\dL(tL^{-3/2}) \in (-M_3,M_2) \ \big| \ \dL(1)\ge L \right) \\
    &\ge 1- \prob \left(|L \pi (tL^{-3/2})| \ge  M_1 \ \big| \ \dL(1)\ge L \right)-\prob\left(  L^{1/2}   \dL(tL^{-3/2})  > M_2 \ \big| \ \dL(1)\ge L \right) \\
    &\quad - \prob\left( L \pi (tL^{-3/2}) \in (-M_1, M_1), L^{1/2}   \dL(tL^{-3/2})  <- M_3 \ \big| \ \dL(1)\ge L \right)\\
    & \ge 1-3\epsilon
    \end{split}
\end{equation}
for sufficiently large $L$. By taking $L\to\infty$ and applying Theorem \ref{thm:main}, we obtain 
\begin{equation}
     \int_{-M_3}^{M_2} \int_{-M_1}^{M_1} \rp(h,x;t) \rd x\rd h\ge 1-3\epsilon,
\end{equation}
which implies 
\begin{equation}
    \int_{\realR} \int_\realR \rp(h,x;t) \rd x\rd h\ge 1-3\epsilon.
\end{equation}
Note that $\epsilon$ could be arbitrarily small. Thus, $\iint_{\realR^2}\rp(h,x;t) \rd x\rd h \ge 1 $. On the other hand, for arbitrary $M_1',M_2'>0$
\begin{equation}
    \int_{-M_2'}^{M_2'} \int_{-M_1'}^{M_1'} \rp(h,x;t) \rd x\rd h = \lim_{L\to\infty}\prob\left(L \pi (tL^{-3/2}) \in (-M_1',M_1'), L^{1/2}\dL(tL^{-3/2}) \in (-M_2',M_2') \ \big| \ \dL(1)\ge L \right) ,
\end{equation}
which is the limit of a sequence of probabilities that are bounded by $1$. Thus, $\int_{-M_2'}^{M_2'} \int_{-M_1'}^{M_1'} \rp(h,x;t) \rd x\rd h\le 1$ for any  $M_1',M_2'>0$. This further implies $\int_{\realR}\int_\realR\rp(h,x;t) \rd x\rd h \le 1 $. Combining the arguments above, we obtain the desired statement $\int_{\realR}\int_\realR\rp(h,x;t) \rd x\rd h = 1 $.

It remains to prove the two lemmas \ref{lm:bound_geodesic_location} and \ref{lm:bound_geodesic_length}.

\subsection{Bound of the conditional geodesic location}
\label{sec:geodecis_location}

In this subsection, we use a probabilistic argument to prove Lemma \ref{lm:bound_geodesic_location}. Note that $\pi(\tau)$ and $-\pi(\tau)$ have the same distribution. Thus, it is sufficient to show that there exists $M_1$ such that 
\begin{equation}
    \label{eq:onesidebound_geodesic_location}
    \prob\left( \pi_{0,0;0,1} (tL^{-3/2}) \ge M_1L^{-1}  \ \big| \ \dL(0,0;0,1)\ge L\right) \le \epsilon/2
\end{equation}
for sufficiently large $L$. Here we use the notation $\pi_{0,0;0,1}(\tau) $ and $\dL(0,0;0,1)$, instead of their abbreviations $\pi(\tau)$ and $\dL(1)$, to avoid notation confusions since we will consider other geodesics $\pi_{x,s;y,t}$ and directed landscape values $\dL(x,s;y,t)$.

We need the following simple lemma.
\begin{lm}
    \label{lm:simple_lemma}
Assume that $0<\tau<1$, $m>0$, $\ell,\ell',\ell''\in\realR$. We have 
\begin{equation}
    \begin{split}
        &\prob\left(\pi_{0,0;0,1}(\tau) \ge m, \dL(0,0;0,1)\ge \ell\right)\\
        &\le \prob\left(\dL(0,\tau;0,1)\le \ell', \dL(0,0;0,1)\ge \ell\right) + \prob\left(\dL(m,0;0,1)\ge \ell'',\dL(0,0;0,1)\ge \ell\right)\\
        &\quad + \prob\left(\dL(0,\tau;0,1)> \ell'\right)\cdot \prob\left(\dL(0,0;m,\tau)-\dL(0,0;0,\tau)\ge \ell-\ell''\right).
    \end{split}
\end{equation}    
\end{lm}
\begin{proof}[Proof of Lemma \ref{lm:simple_lemma}]
    Note that the event $\{\pi_{0,0;0,1}(\tau) \ge m, \dL(0,0;0,1)\ge \ell\}$ is a subset of the union of three events $\{\dL(0,\tau;0,1)\le \ell', \dL(0,0;0,1)\ge \ell\}$, $\{\dL(m,0;0,1)\ge \ell'',\dL(0,0;0,1)\ge \ell\}$, and $A$, where 
    \begin{equation}
        A =\{\pi_{0,0;0,1}(\tau)\ge m, \dL(0,0;0,1)\ge \ell, \dL(0,\tau;0,1)> \ell',\dL(m,0;0,1)< \ell'' \} .
    \end{equation}
    Thus, it is sufficient to show that 
    \begin{equation}
        \label{eq:inequality_probA}
        \prob(A) \le \prob\left(\dL(0,\tau;0,1)> \ell'\right)\cdot \prob\left(\dL(0,0;m,\tau)-\dL(0,0;0,\tau)\ge \ell-\ell''\right).
    \end{equation}
    Now denote the geodesic from $(m,0)$ to $(m,\tau)$ by $\pi'(\tau') = \pi_{m,0;m,\tau}(\tau')$, $0\le \tau'\le \tau$. Note that $\pi_{0,0;0,1}(\tau)\ge m$ implies that there exists some $0< \tau'\le \tau$, such that $\pi_{0,0;0,1}(\tau')=\pi'(\tau')$, and hence
    \begin{equation}
        \begin{split}
            &\dL(m,0;0,1) - \dL(0,0;0,1) \\
            &\ge (\dL(m,0;\pi'(\tau'),\tau') + \dL(\pi_{0,0;0,1}(\tau'),\tau';0,1) )- (\dL(0,0;\pi'(\tau'),\tau')+ \dL(\pi_{0,0;0,1}(\tau'),\tau';0,1) )\\
            &\ge -\max_{\tau'\in(0,\tau]} \left\{\dL(0,0;\pi'(\tau'),\tau')-\dL(m,0;\pi'(\tau'),\tau')\right\}\\
            &=-(\dL(0,0;\pi'(\tau)=m,\tau)-\dL(m,0;\pi'(\tau)=m,\tau)),
        \end{split}        
    \end{equation}
    where the last inequality follows from the following monotonicity of $\dL(0,0;\pi'(\tau'),\tau')-\dL(m,0;\pi'(\tau'),\tau')$
    \begin{equation}
    \begin{split}
    &\dL(0,0;\pi'(\tau'),\tau')-\dL(m,0;\pi'(\tau'),\tau') \\
    &\ge \dL(0,0;\pi'(\tau'-\epsilon),\tau'-\epsilon) +\dL(\pi'(\tau'-\epsilon),\tau'-\epsilon;\pi'(\tau'),\tau')-\dL(m,0;\pi'(\tau'),\tau')\\
    &= \dL(0,0;\pi'(\tau'-\epsilon),\tau'-\epsilon)- \dL(m,0;\pi'(\tau'-\epsilon),\tau'-\epsilon),\quad 0\le \epsilon<\tau'.
    \end{split}
    \end{equation}
    We remark that similar monotonicity for semi-infinite geodesic was obtained in \cite[Lemma 16]{Bhatia24}.

    As a result, $A\subseteq B=B_1\cap B_2$, where
    \begin{equation}
        B_1= \left\{ \dL(0,0;m,\tau) -\dL(m,0;m,\tau) \ge \ell-\ell''\right\},\quad B_2=\{\dL(0,\tau;0,1)> \ell'\},
    \end{equation}
    and we further have 
    \begin{equation}
        \prob(A) \le \prob(B_1\cap B_2)=\prob(B_1)\prob(B_2),
    \end{equation}
    where the last equation comes from the independence of increments of the directed landscape in disjoint time windows. Finally, using the following symmetries of the law of the directed landscape (see Lemma 10.2 of \cite{Dauvergne-Ortmann-Virag22})
    \begin{equation}
    \begin{split}(\dL(0,0;m,\tau),\dL(m,0;m,\tau))&\dequal (\dL(-m,0;0,\tau),\dL(0,0;0,\tau))\\
    &\dequal (\dL(0,-\tau;m,0),\dL(0,-\tau;0,0))\\
    &\dequal (\dL(0,0;m,\tau),\dL(0,0;0,\tau)),
    \end{split}
    \end{equation}
    we have
    \begin{equation}
        \prob(B_1) = \prob\left(\dL(0,0;m,\tau)-\dL(0,0;0,\tau)\ge \ell-\ell''\right).
    \end{equation}
    The inequality \eqref{eq:inequality_probA} follows immediately. This completes the proof of the lemma.
\end{proof}

Now we choose $\tau=tL^{-3/2}$, $m=M_1L^{-1}$, $\ell=L$, $\ell'= L - M'_1L^{-1/2}$ and $\ell''=L - M_1L^{-1/2}$ in Lemma \ref{lm:simple_lemma}, where $M_1$ and $M_1'$ are constants to be chosen later. We have 
\begin{equation}
    \label{eq:conditionalGeodesicLocation_split}
    \begin{split}
        &\prob\left(\pi_{0,0;0,1} (tL^{-3/2}) \ge M_1L^{-1}  \ \big| \ \dL(0,0;0,1)\ge L\right)\\
        &\le \prob\left(\dL(0,tL^{-3/2};0,1)\le L -M'_1L^{-1/2}\ \big| \ \dL(0,0;0,1)\ge L\right) \\
        &\quad + \prob\left(\dL(M_1L^{-1},0;0,1)\ge L -M_1L^{-1/2}\ \big| \ \dL(0,0;0,1)\ge L\right)\\
        &\quad + \prob\left(\dL(0,tL^{-3/2};0,1)> L -M'_1L^{-1/2}\right)/\prob\left( \dL(0,0;0,1)\ge L\right)\\
        &\qquad\cdot \prob\left(\dL(0,0;M_1L^{-1},tL^{-3/2})-\dL(0,0;0,tL^{-3/2})\ge M_1L^{-1/2}\right).
    \end{split}
\end{equation}
Recall \eqref{eq:UpperTailFieldConvergence} and the relation \eqref{eq:relation_two_UTfields}. Also note that the grounded upper tail field at time $0$, $\UTfield(x,0)$ has the same distribution as the shifted two-sided Brownian motion $\bfB_{\mathrm{ts}}(2x)-2|x|$. We have 
\begin{equation}
\label{eq:conditional_lower_tail_GaussianBound}
    \begin{split}
    \prob\left(\dL(0,tL^{-3/2};0,1)\le L -M'_1L^{-1/2}  \ \big| \ \dL(0,0;0,1)\ge L\right)&\to \prob\left(\LUTfield(0,-t)\le -M'_1\right),\\
    \prob\left(\dL(M_1L^{-1},0;0,1)\ge L -M_1L^{-1/2}\ \big| \ \dL(0,0;0,1)\ge L\right) &\to \prob\left(\LUTfield(M_1,0)\ge -M_1\right)\\
    &=\int_0^\infty 2e^{-2\alpha}\Phi\left(-\sqrt{\frac{M_1}{2}}+\frac{\alpha}{\sqrt{2M_1}}\right) \rd\alpha,
    \end{split}
\end{equation}
where $\Phi$ is the cumulative distribution function of the standard normal distribution. Moreover, by the upper tail approximation of the GUE Tracy-Widom distribution, we have
\begin{equation}
    \prob\left(\dL(0,tL^{-3/2};0,1)> L -M'_1L^{-1/2}\right)/\prob\left( \dL(0,0;0,1)\ge L\right) \to e^{2(M'_1-t/3)}
\end{equation}
for sufficiently large $L$. Finally, using the scaling invariance of the directed landscape and the KPZ fixed point, we have
\begin{equation}
    \label{eq:Airy2_bound}
    \begin{split}
        &\prob\left(\dL(0,0;M_1L^{-1},tL^{-3/2})-\dL(0,0;0,tL^{-3/2})\ge M_1L^{-1/2}\right)\\
        &=\prob\left(\dL(0,0;t^{-2/3}M_1,1)-\dL(0,0;0,1) \ge t^{-1/3}M_1 \right)\\
        &=\prob\left(\cA_2(t^{-2/3}M_1)-\cA_2(0) \ge t^{-4/3}M_1^2+t^{-1/3}M_1\right),
    \end{split}
\end{equation}
where $\cA_2$ is the Airy$_2$ process defined in \cite{Prahofer-Spohn02}. Note that $\cA_2(t^{-2/3}M_1)$ and $\cA_2(0)$ become two independent and identically distributed GUE Tracy-Widom random variables as $M_1$ goes to infinity, see \cite{Adler-van_Moerbeke03,Widom04,Adler-van_Moerbeke05}. Thus, \eqref{eq:Airy2_bound} goes to zero as $M_1$ becomes large. Now we choose $M'_1$ such that $\prob\left(\LUTfield(0,-t)\le - M'_1\right)=\epsilon/8$, and choose $M_1$ such that the right-hand sides of \eqref{eq:conditional_lower_tail_GaussianBound}  and \eqref{eq:Airy2_bound} are less than $\epsilon/8$ and $e^{-2(M'_1-t/3)}\cdot \epsilon/8$ respectively. Using \eqref{eq:conditionalGeodesicLocation_split} and combining the above asympotics, we know that for sufficiently large $L$,
\begin{equation}
    \prob\left(\pi_{0,0;0,1} (tL^{-3/2}) \ge M_1L^{-1}  \ \big| \ \dL(0,0;0,1)\ge L\right) < \frac{\epsilon}{7} + \frac{\epsilon}{7} + \frac{\epsilon}{7} <\frac{\epsilon}{2}.
\end{equation}
This completes the proof of \eqref{eq:onesidebound_geodesic_location}.

\subsection{Bound of the conditional tail probabilities of the geodesic length}
\label{sec:geodesic_length}

In this subsection, we prove Lemma \ref{lm:bound_geodesic_length}. Note that $M_1$ and $\epsilon$ are both fixed.

We first prove part (1). We assume $M_2$ and $M'_2$ are two large parameters to be decided later.

Note that the event $\{L^{1/2}\dL(0,0;\pi_{0,0;0,1}(tL^{-3/2}),tL^{-3/2})\ge M_2, \dL(0,0;0,1)\ge L\}$ is a subset of $B'_1\cup B'_2$, where 
\begin{equation}
    \begin{split}
    B'_1&=\left\{\dL(0,tL^{-3/2};0,1) \le L- M'_2L^{-1/2}, \dL(0,0;0,1)\ge L\right\}, \text{ and }\\
    B'_2&=\left\{\dL(0,tL^{-3/2};0,1) > L- M'_2L^{-1/2}, \max_{y\in\realR}L^{1/2}\dL(0,0;yL^{-1},tL^{-3/2})\ge M_2\right\}.
    \end{split}
\end{equation}
Also note that the events $\{\dL(0,tL^{-3/2};0,0) > L- M'_2L^{-1/2}\}$ and $\{\max_{y\in\realR}L^{1/2}\dL(0,0;yL^{-1},tL^{-3/2})\ge M_2\}$ are independent.
Therefore,
\begin{equation}
    \begin{split}
        \prob(B'_2)
        &= \prob\left(\dL(0,tL^{-3/2};0,1) > L- M'_2L^{-1/2}\right) \prob\left(\max_{y\in\realR}L^{1/2}\dL(0,0;yL^{-1},tL^{-3/2})\ge M_2\right)\\
        &= \prob\left(\dL(0,tL^{-3/2};0,1) > L- M'_2L^{-1/2}\right) \prob\left(\max_{y\in\realR} \dL(0,0;y ,t )\ge M_2\right),
    \end{split}
\end{equation}
where we used the $1:2:3$ scaling invariance property of the directed landscape in the last equation.
We further obtain 
\begin{equation}
    \label{eq:bound_conditionalLengthSmall}
    \begin{split}
        & \prob\left(L^{1/2}\dL(0,0;\pi_{0,0;0,1}(tL^{-3/2}),tL^{-3/2})\ge M_2\ \big| \  \dL(0,0;0,1)\ge L\right)\\
        & \le  \prob\left(\dL(0,tL^{-3/2};0,1) \le L- M'_2L^{-1/2}\ \big| \  \dL(0,0;0,1)\ge L\right)\\
        &\quad +\frac{\prob\left(\dL(0,tL^{-3/2};0,1) > L- M'_2L^{-1/2}\right)}{\prob\left(\dL(0,0;0,1)\ge L\right)}\prob\left(\max_{y\in\realR} \dL(0,0;y ,t )\ge M_2\right).
    \end{split}
\end{equation}
Similarly to the discussions in the previous subsection, the right hand side of the inequality above converges to, as $L\to\infty$, 
\begin{equation}
    \prob\left(\LUTfield(0,-t)\le -M'_2\right) + e^{2(M'_2-t/3)}\prob\left(\max_{y\in\realR} \dL(0,0;y ,t )\ge M_2\right),
\end{equation}
each of which can be made smaller than $\epsilon/3$ by choosing appropriate parameters $M'_2$ and $M_2$. This implies the left hand side of \eqref{eq:bound_conditionalLengthSmall} is smaller than $\epsilon$ for sufficiently large $L$. Lemma \ref{lm:bound_geodesic_length} part (1) follows immediately.

\bigskip
Now we consider the second part of Lemma \ref{lm:bound_geodesic_length}. The following lemma will be used.
\begin{lm}
    \label{lm:max_Airy}
    Assume that $m>0$, $0<\tau<1$, $\ell,\ell'\in\realR$. We have
    \begin{equation}
        \prob\left(\max_{y\in [-m,m]} \dL(y,\tau;0,1)\ge \ell\right)\cdot \left(1-\FGUE\left(\ell'\tau^{-1/3} + m^2\tau^{-4/3}\right)\right)\le \prob \left( \dL(0,0;0,1)\ge \ell+\ell' \right).
    \end{equation}
\end{lm}

\begin{proof}[Proof of Lemma \ref{lm:max_Airy}]
    Denote $y_0=\argmax_{y\in [-m,m]}\dL(y,\tau;0,1)$ which is almost surely well-defined, see \cite{Corwin-Hammond14,Rahman-Virag25}. We consider the event $A'_1\cap A'_2$, where 
    \begin{equation}
        A'_1= \{\dL(y_0,\tau;0,1)\ge \ell\},\quad A'_2=\{\dL(0,0;y_0,\tau)\ge \ell'\}.
    \end{equation}
    Note that for any given $y_0$, $\dL(0,0;y_0,\tau)$ is a shifted and scaled GUE Tracy-Widom random variable due to the independence of $\{\dL(y,\tau;0,1);y\in[-m,m]\}$ and $\{\dL(0,0;y,\tau):y\in[-m,m]\}$. Thus, we have
    \begin{equation}
        \prob\left(A'_2\right) = 1-\FGUE\left(\ell'\tau^{-1/3} + y_0^2\tau^{-4/3}\right) \ge 1-\FGUE\left(\ell'\tau^{-1/3} + m^2\tau^{-4/3}\right).
    \end{equation} 
    Moreover, we have
    \begin{equation}
        \begin{split}
        \prob(A_1'\cap A_2') &\ge \prob(A'_1)\left(1-\FGUE\left(\ell'\tau^{-1/3} + m^2\tau^{-4/3}\right)\right)\\
        &=\prob\left(\max_{y\in [-m,m]} \dL(y,\tau;0,1)\ge \ell\right)\cdot \left(1-\FGUE\left(\ell'\tau^{-1/3} + m^2\tau^{-4/3}\right)\right).
        \end{split}
    \end{equation}
    Lemma \ref{lm:max_Airy} follows by noting $A_1'\cap A'_2$ is a subset of $\{\dL(0,0;0,1)\ge \ell+\ell'\}$.
\end{proof}

Now we choose $\tau=tL^{-3/2}$, $m=M_1L^{-1}$, $\ell= L +M_3L^{-1/2}$ and $\ell'=-\frac12M_3L^{-1/2}$ in Lemma \ref{lm:max_Airy}, where $M_1$ is a fixed positive constant and $M_3$ is a fixed large parameter such that
\begin{equation}
\label{eq:def_M3}
    \frac{2e^{-M_3}}{1-\FGUE\left(-\frac12M_3 t^{-1/3}+ M_1^2t^{-4/3}\right)} \le \epsilon.
\end{equation}
We have
\begin{equation}
    \begin{split}
        &\prob\left(\max_{|Ly|\le M_1 } \dL(y,tL^{-3/2};0,1)\ge L+M_3L^{-1/2} \right)\\
        &\le \frac{\prob\left(\dL(0,0;0,1)\ge L+\frac12M_3L^{-1/2}\right)}{1-\FGUE\left(-\frac12M_3 t^{-1/3}+ M_1^2t^{-4/3}\right)}.
    \end{split}
\end{equation}
Note that, using \eqref{eq:GUE_asymp},
\begin{equation}
    \frac{\prob\left(\dL(0,0;0,1)\ge L+\frac12M_3L^{-1/2}\right)}{\prob\left(\dL(0,0;0,1)\ge L\right)}\to e^{-M_3}
\end{equation}
as $L\to\infty$. Inserting it into the previous inequality, we get 
\begin{equation}
    \begin{split}
        &\prob\left(\max_{|Ly|\le M_1 } \dL(y,tL^{-3/2};0,1)\ge L+M_3L^{-1/2} \ \big| \  \dL(0,0;0,1)\ge L\right) \\
        &\le \frac{2e^{-M_3}}{1-\FGUE\left(-\frac12M_3 t^{-1/3}+ M_1^2t^{-4/3}\right)}
    \end{split}    
\end{equation}
for sufficiently large $L$. This bound is less than $\epsilon$ by our assumption on $M_3$, see \eqref{eq:def_M3}. Lemma \ref{lm:bound_geodesic_length} (2) follows by noting that $\{ L \pi (tL^{-3/2}) \in (-M_1, M_1), L^{1/2}   \dL(tL^{-3/2})  <- M_3, \dL(1)\ge L \}$ is a subset of $\{\max_{|Ly|\le M_1 } \dL(y,tL^{-3/2};0,1)\ge L+M_3L^{-1/2},\dL(0,0;0,1)\ge L\}$.

\section{Proof of Proposition \ref{prop:properties_p}}
\label{sec:proof_proposition2}

Proposition \ref{prop:properties_p} part (1) follows from a simple symmetry of the formulas \eqref{eq:def_hat_rI} and \eqref{eq:rI_explicit}. In fact, if we do a change of variable $u_i=-v'_i$ and $v_i=-u'_i$, then the formulas \eqref{eq:def_hat_rI} and \eqref{eq:rI_explicit} retain the same structure, except that $u_i$ and $v_i$ are replaced by $u'_i$ and $v'_i$ respectively, and that $x$ is replaced by $-x$. This implies $\rp(h,x;t)=\rp(h,-x;t)$ and $\hat\rp(h,x;t)=\hat\rp(h,-x;t)$.

For part (2), the first identity \eqref{eq:relation_two_densities1} follows from the definition of $\hat\rp$, see \eqref{eq:def_hat_rp01}. The second identity \eqref{eq:relation_two_densities2} is a corollary of \eqref{eq:relation_two_densities1}.

The proof of the other two parts of Proposition \ref{prop:properties_p} relies on asymptotic analysis, and is provided in the following two subsections. We also mention the following two simple identities about the function $\rH$ defined in \eqref{eq:def_rH}, which will be used in both subsections.
\begin{equation}
\label{eq:rH_identity1}
    \rH(u,1;v,-1) = 2(1-v)(1+u)(u-v)
\end{equation}
and
\begin{equation}
\label{eq:rH_identity2}
    \rH(u,-1,1;v,1,-1)=\rH(u;v)=0.
\end{equation}

\subsection{Large time limit of the joint density function $\rp$ and proof of \eqref{eq:large_time_asymptotics}}
\label{sec:large_time_limit}

We first remark that the statement \eqref{eq:large_time_asymptotics} implies the same statement for $\hat\rp$ by using Proposition \ref{prop:properties_p} (2). Thus the third part of this proposition is reduced to the statement for $\rp$, which will be proved in this subsection.

The proof of \eqref{eq:large_time_asymptotics} is very similar to that of the Theorem 1.1 in \cite{Liu22b}, more concretely, the proof of the approximation (3.7) in \cite{Liu22b}. Although the parameters are scaled differently, the formulas are the same, the leading contributions both come from the first term in the expansion, and the estimates can be handled in a similar way in this paper and \cite{Liu22b}. Hence, we only provide the ideas of the proof and omit the details here.

Our goal is to prove that 
\begin{equation}
\label{eq:LargeTime_rp}
\rp\left(t+h\sqrt{t},\frac12x\sqrt{t};t\right) = \frac{2}{t}(\phi(h)\phi(x) +o(1))
\end{equation}
as $t\to\infty$. Here $h$ and $x$ are fixed real numbers. Recall the definition of $\rp$ and the equation \eqref{eq:def_rp}. The dominating factor, the $f$ function, becomes
\begin{equation}
f(w) =\exp\left( -\frac13tw^3 +\frac12x\sqrt{t}w^2 +(t+h\sqrt{t})w\right).
\end{equation}
Note that when $t$ becomes large, the critical points for $f(w)$ are $w=-1,1$. If we deform the contours $u_i$ and $v_i$ near the critical points, we get $f(u_i)\approx f(-1)=\exp\left(-\frac{2}{3}t+o(t)\right)$ and $f(v_i)\approx f(1)=\exp\left(\frac23t+o(t)\right)$.  Using this fact, we see that the leading contribution of the summation \eqref{eq:def_rp} comes from the term when $n=1$, and other terms with larger $n$ values give exponentially smaller contributions. 

Now we  evaluate the contribution from the term with $n=1$ in the formula \eqref{eq:def_rp}. When $n=1$, we denote $\vec{U}=(u)$ and $\vec{V}=(v)$ and obtain, using \eqref{eq:rI_explicit} and \eqref{eq:rH_identity1},
\begin{equation}
    \rI_{n=1}(\vec{U};\vec{V})=-2e^{\frac43t + 2h\sqrt{t}} \frac{(1+v)(1-u)}{u-v} \frac{f(u)}{f(v)}.
\end{equation}
Hence heuristically
\begin{equation}
\label{eq:LargeTime_rp_leading_term}
    \rp\left(t+h\sqrt{t},\frac12x\sqrt{t};t\right)\approx -2e^{\frac43t + 2h\sqrt{t}}\int_{\Gamma_\rL} \int_{\Gamma_\rR} \frac{(1+v)(1-u)}{u-v} \frac{f(u)}{f(v)} \frac{\rd u}{2\pi\rmi} \frac{\rd v}{2\pi\rmi}.
\end{equation}
Note that although originally the $u$-contour $\Gamma_\rL$ and $v$-contour $\Gamma_\rR$ have certain restrictions on their intersection with the real axis, as shown in Figure \ref{fig:Gamma_Contours}, this restriction can be removed in the formula above since the integrand is analytic at $u=-1$ and $v=1$.

We apply the standard asymptotic analysis for the integrals in \eqref{eq:LargeTime_rp_leading_term}. As we mentioned before, the leading contribution of the integrals comes from $u\approx -1$ and $v\approx 1$. More explicitly, we zoom in the neighborhood of $-1$ and $1$ and rescale the $u$, $v$ variables as follows
\begin{equation}
    u= -1 +\frac{\xi}{\sqrt{2t}},\quad v= 1+\frac{\eta}{\sqrt{2t}},
\end{equation}
for $\xi,\eta\in\rmi\realR$. Then 
\begin{equation}
    f(u)\approx e^{-\frac23t -h\sqrt{t} +\frac12 x\sqrt{t} +\frac12\xi^2 -\frac{x-h}{\sqrt{2}}\xi},\quad f(v)\approx e^{\frac23t +h\sqrt{t} +\frac12 x\sqrt{t} -\frac12\eta^2 +\frac{x+h}{\sqrt{2}}\eta}
\end{equation}
and
\begin{equation}
    \frac{(1+v)(1-u)}{u-v}\approx -2.
\end{equation}
As a result, we have
\begin{equation}
    \rp\left(t+h\sqrt{t},\frac12x\sqrt{t};t\right)\approx \frac{2}{t}\int_{\rmi\realR}\int_{\rmi\realR} e^{\frac12\xi^2 -\frac{x-h}{\sqrt{2}}\xi +\frac12\eta^2 -\frac{x+h}{\sqrt{2}}\eta}\frac{\rd \xi}{2\pi\rmi}\frac{\rd \eta}{2\pi\rmi}=\frac{1}{\pi t} e^{-\frac{(x-h)^2+(x+h)^2}{4}} 
\end{equation}
which implies \eqref{eq:LargeTime_rp}. 

Note that our heuristic arguments can be made rigorous by the standard asymptotic analysis and the dominated convergence theorem, as shown in \cite{Liu22b}.

\subsection{Approximation of the joint density functions $\rp$ and $\hat\rp$ when $h$ or $|x|$ goes to infinity}
\label{sec:right_tail}

In this subsection, we fix $t>0$ and consider the asymptotics of $\rp(h,x;t)$ and $\hat\rp(h,x;t)$ as one of $h$ and $x$ goes to infinity while the other is fixed. Note that the formula of $\rp(h,x;t)$ has a very similar structure as that of the one point distribution of the lifted upper tail field $\prob\left(\LUTfield(x,-t) \ge -h\right)$ as described in Theorem \ref{thm:connection_rp_UTfield}. When $x$ and $t$ are fixed, the asymptotics of $\prob\left(\LUTfield(x,-t) \ge -h\right)$ when $h\to\infty$ was obtained in \cite[Proposition 2.9]{Liu-Zhang25}. Although the formula of $\rp(h,x;t)$ contains an extra factor in the integrand (see the last factor in \eqref{eq:rI_explicit}) which indeed affects the leading contribution term, the idea of the asymptotic remains similar. For this reason, we will focus on the main ideas to prove \eqref{eq:right_tail_rp_h} and \eqref{eq:right_tail_rp_x}, in particular the differences compared to the work in \cite{Liu-Zhang25}, and omit the details which are similar to that paper.

\medskip

Now we consider \eqref{eq:right_tail_rp_h} and \eqref{eq:right_tail_rp_x}. Due to the symmetry of $\rp$ and $\hat{\rp}$ in Proposition \ref{prop:properties_p} (1), we always assume that $x\ge 0$ in this subsection. For either case to be proved, we have
\begin{equation}
\label{eq:def_M}
M := t^{-1/6} \left( h + \frac{x^2}{t}\right)^{1/2}
=\begin{dcases}
t^{-1/6}h^{1/2}(1+O(h^{-1})), & \text{ when $x$ fixed and $h\to\infty$},\\
t^{-2/3}x (1+O(x^{-2})),& \text{ when $h$ fixed and $x\to\infty$}.
\end{dcases}
\end{equation}

\medskip

We first focus on the asymptotic analysis of $\rp$. The case for $\hat\rp$ is similar and will be discussed at the end of this subsection.

Recall Definition \ref{def:rp}. We have
\begin{equation}
\label{eq:recall_rp}
\rp(h,x;t) =\sum_{n\ge 1} \frac{1}{(n!)^2} \int_{\Gamma_\rL^n\times\Gamma_\rR^n} \rI_n(\vec{U};\vec{V}) \prod_{i=1}^n \frac{\rd u_i}{2\pi\rmi}\prod_{i=1}^n \frac{\rd v_i}{2\pi\rmi}
\end{equation}
with
\begin{equation}
\label{eq:recall_rI}
\begin{split}
\rI_n(\vec{U};\vec{V})
& = (-1)^{n}e^{-\frac{2}{3}t+2h}
\frac{\prod_{1\le i<j\le n}(u_i-u_j)^2(v_i-v_j)^2}{\prod_{i=1}^n\prod_{j=1}^n (u_i-v_j)^2}   \prod_{i=1}^n \frac{(1+v_i)(1-u_i)f_{h,x;t}(u_i)}{(1+u_i)(1-v_i)f_{h,x;t}(v_i)}\\
&\quad \cdot \rH(\vec{U}\sqcup(1);\vec{V}\sqcup(-1)).
\end{split}
\end{equation}
Note that the parameters $h,x,t$ appear only in the exponential functions $e^{-\frac23t+2h}$ and $f_{h,x;t}$. Moreover, one can directly write
\begin{equation}
\begin{split}
    f_{h,x;t}(w) = e^{-\frac{t}{3}w^3 +xw^2+ hw}
    =e^{M^{3}\left(-\frac{1}{3}z^3 +z\right)} \cdot e^{\frac{2x^3}{3t^2}+\frac{hx}{t}},
\end{split}
\end{equation}
where 
\begin{equation}
z=\frac{t^{1/3}}{M}\left(w-\frac{x}{t}\right).
\end{equation}
The function $-\frac13z^3+z$ has two critical points $1$ and $-1$. Denote the corresponding $w$ values by
\begin{equation}
w_c^- = \frac{x}{t} - t^{-1/3}M,\quad \text{and}\quad w_c^+ = \frac{x}{t} + t^{-1/3}M.
\end{equation}
Note that $w_c^- < 1$ and $w_c^+ > 0$ under our assumptions that $x > 0$ and $M \to \infty$, in both cases we aim to prove.

When $M\to\infty$, the standard steepest descent method gives
\begin{equation}
\label{eq:steepestDescent_oneterm}
    \int_{\Gamma_\rL} f_{h,x;t}(u) \frac{\rd u}{2\pi\rmi} \approx e^{\frac{hx}{t} +\frac{2x^3}{3t^2}} \cdot e^{-\frac{2}{3}M^{3}+O(\log M)}, \text{ and } \int_{\Gamma_\rR} \frac{1}{f_{h,x;t}(v)} \frac{\rd v}{2\pi\rmi} \approx e^{-\frac{hx}{t} -\frac{2x^3}{3t^2}} \cdot e^{-\frac{2}{3}M^{3}+O(\log M)}
\end{equation}
by deforming the $\Gamma_\rL$ and $\Gamma_\rR$ contours close to the point $w_c^-$ and $w_c^+$ respectively, where the main contribution comes from. Therefore, if we ignore all other factors in $\rI_n$, we have
\begin{equation}
\int_{\Gamma_\rL^n\times\Gamma_\rR^n} \prod_{i=1}^n\frac{f_{h,x;t}(u_i)}{f_{h,x;t}(v_i)}\prod_{i=1}^n \frac{\rd u_i}{2\pi\rmi}\prod_{i=1}^n \frac{\rd v_i}{2\pi\rmi}\approx e^{-\frac{4n}3M^{3}},
\end{equation}
which decreases as $n$ increases. Note that other factors in $\rI_n(\vec{U};\vec{V})$, except for the factor $e^{-\frac23t+2h}$ which does not depend on the $u_i,v_i$ variables, are independent of the large parameters. Moreover, when we deform all the $u_i$ contours $\Gamma_\rL$ close to $w_c^-$, at most one pole at $u_i=-1$ could make a nonzero contribution due to the factor $\prod_{i<j}(u_i-u_j)^2$. Similarly for the $v_i$  contours $\Gamma_\rR$. Therefore, heuristically the first two terms in the expansion \eqref{eq:recall_rp} would dominate the summation in both cases, and we have
\begin{equation}
\label{eq:approx_rp_two_terms}
    \rp(h,x;t)\approx \int_{\Gamma_\rL\times\Gamma_\rR} \rI_1(u,v)\frac{\rd u }{2\pi\rmi}\frac{\rd v}{2\pi\rmi} + \frac{1}{4} \int_{\Gamma_\rL^2\times\Gamma_\rR^2} \rI_2(u_1,u_2;v_1,v_2)\frac{\rd u_1}{2\pi\rmi}\frac{\rd u_2}{2\pi\rmi}\frac{\rd v_1}{2\pi\rmi}\frac{\rd v_2}{2\pi\rmi}.
\end{equation}
Below we approximate the two terms on the right hand.

\bigskip

For the first term, similar to the $n=1$ case we discussed in the previous subsection, we have
\begin{equation}
    \int_{\Gamma_\rL\times\Gamma_\rR} \rI_1(u,v)\frac{\rd u }{2\pi\rmi}\frac{\rd v}{2\pi\rmi} = -2e^{-\frac23t +2h} \int_{\Gamma_\rL\times\Gamma_\rR} \frac{(1+v)(1-u)}{u-v} \frac{f_{h,x;t}(u)}{f_{h,x;t}(v)} \frac{\rd u }{2\pi\rmi}\frac{\rd v}{2\pi\rmi}.
\end{equation}
We apply the steepest descent analysis to the double contour integral. The main contribution of the integral comes from the neighborhoods of $u\approx w_c^-$ and $v\approx w_c^+$. We write $u=w_c^-+ t^{-1/3}\xi M^{-1/2}$, $v=w_c^++t^{-1/3}\eta M^{-1/2}$, and obtain
\begin{equation}
\begin{split}
    \int_{\Gamma_\rL\times\Gamma_\rR} \frac{(1+v)(1-u)}{u-v} \frac{f_{h,x;t}(u)}{f_{h,x;t}(v)} \frac{\rd u }{2\pi\rmi}\frac{\rd v}{2\pi\rmi} &\approx \frac{(1+w_c^+)(1-w_c^-)e^{-4M^3/3}}{t^{2/3}M (w_c^--w_c^+)} \int_{\rmi\realR} \int_{\rmi\realR} e^{\xi^2+\eta^2}\frac{\rd \xi}{2\pi\rmi}\frac{\rd \eta}{2\pi\rmi}\\
    &\approx \frac{(1+t^{-1/3}M )^2 - x^2t^{-2}}{4\pi (-2t^{1/3}M^2 )}e^{-4M^3/3}\\
    &\approx 
    \begin{dcases}
    -\frac{1}{8\pi t}e^{-4M^3/3}, & \text{ when $x$ is fixed and $h\to\infty$},\\
    -\frac{1}{4\pi x}e^{-4M^3/3}, & \text{ when $h$ is fixed and $x\to\infty$}.
    \end{dcases}
    \end{split}
\end{equation}
Thus we have 
\begin{equation}
\label{eq:leading_term_asymptotics_rp}
    \int_{\Gamma_\rL\times\Gamma_\rR} \rI_1(u,v)\frac{\rd u }{2\pi\rmi}\frac{\rd v}{2\pi\rmi} \approx \begin{dcases}
    \frac{1}{4\pi t}e^{-\frac{4\left(h+\frac{x^2}{t}\right)^{3/2}}{3t^{1/2}}+2h-\frac23t}, & \text{ when $x$ is fixed and $h\to\infty$},\\
    \frac{1}{2\pi x}e^{-\frac{4\left(h+\frac{x^2}{t}\right)^{3/2}}{3t^{1/2}}+2h-\frac23t}, & \text{ when $h$ is fixed and $x\to\infty$}.
    \end{dcases}
\end{equation}

We remark that a similar approximation was done in \cite{Liu-Zhang25} for the double integral of \eqref{eq:recall_rI} for $n=1$ but without the last factor $\rH(u,1;v,-1)=2(1-v)(1+u)(u-v)$ (see \eqref{eq:rH_identity1}) in $\rI_1(u,v)$. The main difference in that paper is that due to this missing factor which would cancel the factors $(1+u)(1-v)$ in the denominator, the leading contribution of the double integral in that paper is obtained when one integral results in the residue at $u=-1$ or $v=1$ while the other integral is deformed to the critical point. This results in a difference of the exponent in the approximation in \cite{Liu-Zhang25}. See the proof of Proposition 2.9 in the paper for more details.

\medskip

We continue to consider the second term in \eqref{eq:approx_rp_two_terms}. We will prove that this term is exponentially smaller compared to the first term we just considered.

We consider the two cases separately. When $x>0$ is fixed and $h\to\infty$, we see that the critical points $w_c^-<-1$ and $w_c^+>1$. Therefore when we deform the $u_1,u_2$ contours to the critical point $w_c^{-}$, we need to pass the pole at $u_1=-1$ or $u_2=-1$. Similarly for the $v_1$ and $v_2$ contours. Note the factors $(u_1-u_2)^2(v_1-v_2)^2$ in $\rI_2(u_1,u_2;v_1,v_2)$. These factors imply that at most one of the $u_i$ variables is evaluated at the residue at $-1$, and at most one of the $v_i$ variables is evaluated at the residue at $1$. Finally, it is easy to verify that if $u_i=-1$ and $v_j=1$ for some $1\le i,j\le 2$, the factor $\rH(u_1,u_2,1;v_1,v_2,-1)$ in $\rI_2(u_1,u_2;v_1,v_2)$ becomes zero, see \eqref{eq:rH_identity2}. Combining the arguments, we can write $\int_{\Gamma_\rL^2\times\Gamma_\rR^2} \rI_2(u_1,u_2;v_1,v_2)\frac{\rd u_1}{2\pi\rmi}\frac{\rd u_2}{2\pi\rmi}\frac{\rd v_1}{2\pi\rmi}\frac{\rd v_2}{2\pi\rmi}$ as the sum of three terms, the first term is a quadruple integral each of which passes a critical point, the second term is a triple integral  each of which passes a critical point and one $u_i$ integral is degenerated to the residue at $u_i=-1$, and the third term is a triple integral each of which passes a critical point and one $v_i$ integral is degenerated to the residue at $v_i=1$. Using the approximation \eqref{eq:steepestDescent_oneterm}, we know that the three terms are approximately
\begin{equation}
      e^{-\frac{8M^3}{3}-\frac{2}{3}t+2h+O(\log M)}, e^{-2M^3 -\frac{hx}{t}-\frac{2x^3}{3t^2} -\frac{1}{3}t+x+h +O(\log M)},\text{ and }e^{-2M^3 +\frac{hx}{t}+\frac{2x^3}{3t^2} -\frac{1}{3}t-x+h +O(\log M)}
\end{equation}
respectively. It is straightforward to verify that each exponent above is significantly less than $-\frac{4M^3}{3}+2h-\frac{2t}{3}$ in the exponent of \eqref{eq:leading_term_asymptotics_rp}.

The second case for $h$ fixed but $x\to\infty$ is similar. The only difference is that the $w_c^-= \frac{x}{t}-t^{-1/3}M = O(x^{-1})$. Thus we could deform the $u_i$ contours to the critical point without encountering any pole. Therefore $\int_{\Gamma_\rL^2\times\Gamma_\rR^2} \rI_2(u_1,u_2;v_1,v_2)\frac{\rd u_1}{2\pi\rmi}\frac{\rd u_2}{2\pi\rmi}\frac{\rd v_1}{2\pi\rmi}\frac{\rd v_2}{2\pi\rmi}$ can be expressed as the sum of two terms instead of three. But the other arguments remain the same and these terms are still exponentially smaller than \eqref{eq:leading_term_asymptotics_rp}.

In conclusion, we obtain the following asymptotics
\begin{equation}
    \rp(h,x;t)\approx \begin{dcases}
    \frac{1}{4\pi t}e^{-\frac{4\left(h+\frac{x^2}{t}\right)^{3/2}}{3t^{1/2}}+2h-\frac23t}, & \text{ when $x$ is fixed and $h\to\infty$},\\
    \frac{1}{2\pi x}e^{-\frac{4\left(h+\frac{x^2}{t}\right)^{3/2}}{3t^{1/2}}+2h-\frac23t}, & \text{ when $h$ is fixed and $x\to\infty$}.
    \end{dcases}
\end{equation}
This proves Proposition \ref{prop:properties_p} (4) for the function $\rp$.

The proof for the function $\hat \rp$ is similar. We note that the only difference is the following factor in $\hat\rI_n$
\begin{equation}
    \frac{2}{\sum_{i=1}^n(v_i-u_i)}
\end{equation}
which does not affect the exponents in the arguments for the integral of $\rI_n$ above. This term when $n=1$ contributes the following factor in the asymptotic analysis
\begin{equation}
    \frac{2}{w_c^+-w_c^-} =\frac{t^{1/3}}{M}.
\end{equation}
This implies $\hat{\rp}(h,x;t)\approx \frac{t^{1/3}}{M}\rp(h,x;t)$ as $M \to \infty$. Note the approximation of $M$ in \eqref{eq:def_M}. The approximations for $\hat\rp$ follow immediately.

\def\cydot{\leavevmode\raise.4ex\hbox{.}}

\end{document}